\documentclass[a4paper,11pt]{article}
\usepackage[utf8]{inputenc}
\usepackage{amsmath,amsfonts}
\usepackage{amsthm}
\usepackage{amssymb}
\usepackage[parfill]{parskip}
\usepackage[english]{babel}
\usepackage[T1]{fontenc}
\usepackage[utf8]{inputenc}
\usepackage{enumerate}
\usepackage{verbatim}
\usepackage{graphicx}
\usepackage{enumitem}
\usepackage{verbatim}
\usepackage[all]{xy}
\usepackage{tikz}
\usepackage{relsize}
\usepackage{amscd}
\usepackage{hyperref}
\usepackage{caption}
\usepackage{xcolor}
\usepackage[all]{xy}
\usepackage{tikz}
\usepackage{relsize}
\usepackage{hyperref}
\usepackage{caption}
\usepackage{xcolor}
\usepackage[hyperpageref]{backref} 
\usepackage{textcomp}

\usetikzlibrary{knots}
\makeatletter
\def\thm@space@setup{%
  \thm@preskip=\parskip \thm@postskip=0pt
}
\makeatother
\newtheorem{teo}{Theorem}[section]
\newtheorem{lem}[teo]{Lemma}
\newtheorem{prop}[teo]{Proposition}
\newtheorem{cor}[teo]{Corollary}
\newtheorem*{teo*}{Theorem}

\newtheorem*{teo:maxvol}{Theorem \ref{maxvol}}
\newtheorem*{teo:cong}{Theorem \ref{teo:cong}}
\newtheorem*{teo:maxvolconj}{Theorem \ref{teovol}}
\newtheorem*{lem:stima3}{Lemma \ref{stima3}}
\newtheorem*{lem:symm}{Lemma \ref{symm}}
\newtheorem*{prop:1}{Proposition \ref{prop:degen}}
\newtheorem*{prop:2}{Proposition \ref{prop:gener}}
\newtheorem*{prop:3}{Proposition \ref{prop:impropideal}}
\theoremstyle{definition}
\newtheorem{dfn}[teo]{Definition}
\newtheorem*{dfn*}{Definition}
\newtheorem{ex}[teo]{Example}
\newtheorem{oss}[teo]{Remark}
\newtheorem*{dom*}{Question}
\newtheorem*{cng:maxvol}{The Maximum Volume Conjecture}

\newcommand{\R}{\mathbb{R}}

\newcommand{\N}{\mathbb{N}}

\newcommand{\ra}{\rightarrow}

\newcommand{\vol}{\mathrm{Vol}}

\newcommand{\h}{\mathbb{H}^3}
\newcommand{\rp}{\mathbb{RP}^3}

\title{The maximum volume of hyperbolic polyhedra}
\author{Giulio Belletti}
\date{}
\newcommand{\address}{{
  \bigskip
  \footnotesize

  Giulio Belletti, \textsc{Scuola Normale Superiore, Pisa, Italy}\par\nopagebreak
  \textit{E-mail address},  \texttt{giulio.belletti@sns.it}

}}
\begin{document}

\maketitle

\begin{abstract}
 We study the supremum of the volume of hyperbolic polyhedra with some fixed combinatorics and with vertices of any kind (real, ideal or hyperideal). We find that the supremum is always equal to the volume of the rectification of the $1$-skeleton. 
 
 The theorem is proved by applying a sort of volume-increasing flow to any hyperbolic polyhedron. Singularities may arise in the flow because some strata of the polyhedron may degenerate to lower-dimensional objects; when this occurs, we need to study carefully the combinatorics of the resulting polyhedron and continue with the flow, until eventually we get a rectified polyhedron.
\end{abstract}
\tableofcontents
\section{Introduction}

In \cite{maxvolconj}, the author introduces a ``Turaev-Viro'' invariant of graphs $\Gamma\subseteq S^3$, denoted with $TV(\Gamma)$, and proposes the following conjecture on its asymptotic behavior in the case where $\Gamma$ is planar and $3$-connected (which is equivalent to saying that $\Gamma$ is the $1$-skeleton of some polyhedron). 

\begin{cng:maxvol}\label{maxvolconj}
 Let $\Gamma\subseteq S^3$ be a $3$-connected planar graph.
 Then
 \begin{displaymath}
  \lim_{r\ra+\infty} \frac{\pi}{r} \log\left(TV_r(\Gamma)\right)=\sup_{P}\mathrm{Vol}(P)
 \end{displaymath}
 where $P$ varies among all proper generalized hyperbolic polyhedra (see Definition \ref{dfn:poly}) with $\Gamma$ as a $1$-skeleton, and $r$ ranges across all odd natural numbers.
\end{cng:maxvol}

The Maximum Volume Conjecture naturally leads to the question of what is the supremum of all volumes of polyhedra sharing the same $1$-skeleton. This is answered here by the following:

\begin{teo:maxvol}
 For any $3$-connected planar graph $\Gamma$,
 $$\sup_{P}\mathrm{Vol}(P)=\vol\left(\overline{\Gamma}\right)$$
 where $P$ varies among all proper generalized hyperbolic polyhedra with $1$-skeleton $\Gamma$ and $\overline{\Gamma}$ is the rectification of $\Gamma$.
\end{teo:maxvol}

The rectification of a graph is defined in Definition \ref{dfn:rect}; for now it suffices to say that $\overline{\Gamma}$ is a finite volume hyperbolic polyhedron that can be easily computed (together with its volume) from $\Gamma$.

This result is obtained by applying a sort of volume-increasing ``flow'' to a polyhedron and carefully analyzing the resulting degenerations.

For the tetrahedron Theorem \ref{maxvol} was proven in \cite[Theorem 4.2]{ushi}. In this case $\sup_{T}\vol(T)$ among all hyperbolic tetrahedra is equal to $v_8\sim3.66$, the volume of the ideal right-angled octahedron. We will see in Section \ref{sec:rect} that this is indeed the volume of the rectification of the tetrahedral graph.

In Section \ref{sec:volconj} we give the basic definitions related to the hyperbolic polyhedra we consider. In Section \ref{sec:background} we describe the space of polyhedra and the properties of the volume function; the material in this section is mostly an expansion of well-known classical results to a wider class of polyhedra. Finally in Section \ref{sec:maxvol} we give the proof of Theorem \ref{maxvol}. 

\textbf{Acknowledgments.} I wish to thank my advisors Francesco Costantino and Bruno Martelli for their constant guidance and support. I would also like to thank the participants of the Geometry Seminar at the University of Pisa, especially Roberto Frigerio and Leone Slavich, for their many helpful comments.

\section{Generalized hyperbolic polyhedra}\label{sec:volconj}

 Recall the projective model for hyperbolic space $\mathbb{H}^3\subseteq\mathbb{R}^3\subseteq\mathbb{RP}^3$ where $\mathbb{H}^3$ is the unit ball of $\R^3$ (for the basic definitions see for example \cite{bonbao}). Notice that for convenience we have picked an affine chart $\mathbb{R}^3\subseteq\mathbb{RP}^3$, so that it always make sense to speak of segments between two points, half spaces, etcetera; this choice is inconsequential, up to isometry. Isometries, in this model, correspond to projective transformations that preserve the unit sphere.
 
We can associate to a point $p$ lying in $\mathbb{R}^3\backslash \overline{\mathbb{H}^3}$ a plane $\Pi_p\subseteq \h$, called the \emph{polar plane} of $p$, such that all lines passing through $\mathbb{H}^3$ and $p$ are orthogonal to $\Pi_p$.
If $p\in \mathbb{R}^3\backslash\overline{\mathbb{H}^3}$, denote with $H_p\subseteq\mathbb{H}^3$ the half space delimited by the polar plane $\Pi_p$ on the other side of $p$; in other words, $H_p$ contains $0\in\R^3$. If $p,p'\in\R^3\backslash\h$ and the line from $p$ to $p'$ passes through $\mathbb{H}^3$, then $\Pi_p$ and $\Pi_{p'}$ are disjoint \cite[Lemma 4]{bonbao}. Specifically, if the segment from $p$ to $p'$ intersects $\mathbb{H}^3$, then $\Pi_p\subseteq H_{p'}$ and $\Pi_{p'}\subseteq H_{p}$; if however the segment does not intersect $\h$, but the half line from $p$ to $p'$ does, then $H_{p}\subseteq H_{p'}$. If $p$ gets pushed away from $\h$, then $\Pi_p$ gets pushed closer to the origin of $\R^3$.
 
 \begin{dfn}
  A \emph{projective polyhedron} in $\mathbb{RP}^3$ is a non-degenerate convex polyhedron in some affine chart of $\mathbb{RP}^3$. Alternatively, it is the closure of a connected component of the complement of finitely many planes in $\mathbb{RP}^3$ that does not contain any projective line.
 \end{dfn}

 Up to isometry of $\h$ we can assume that any projective polyhedron is contained in the standard affine chart.
 
\begin{dfn} \label{dfn:poly}
We introduce the following definitions.
 \begin{itemize}
  \item 
We say that a projective polyhedron $P\subseteq \R^3\subseteq\mathbb{RP}^3$ is a \emph{generalized hyperbolic polyhedron} if each edge of $P$ intersects $\mathbb{H}^3$ (\cite[Definition 4.7]{rivhodg}).
\item A vertex of a generalized hyperbolic polyhedron is \emph{real} if it lies in $\h$, \emph{ideal} if it lies in $\partial \h$ and \emph{hyperideal} otherwise.
\item A generalized hyperbolic polyhedron $P$ is \emph{proper} if for each hyperideal vertex $v$ of $P$ \emph{the interior} of the polar half space $H_v$ contains all the other real vertices of $P$ (see Figure \ref{fig:propbadtrunc}, left). We say that it is \emph{almost proper} if it is not proper but still for each hyperideal vertex $v$ of $P$ , the polar half space $H_v$ contains all the other real vertices of $P$; we call a vertex $v$ belonging to some $\Pi_{v'}$ an \emph{almost proper vertex} (see Figure \ref{fig:propbadtrunc}, right), and $\overrightarrow{vv'}$ an almost proper edge (by contrast, the other vertices and edges are \emph{proper}).
\item 
 We define the \emph{truncation} of a generalized hyperbolic polyhedron $P$ at a hyperideal vertex $v$ to be the intersection of $P$ with $H_v$; similarly the \emph{truncation} of $P$ is the truncation at every hyperideal vertex, that is to say $P\cap\left(\cap_{v \textrm{ hyperideal}}H_v\right)$. We say that the \emph{volume} of $P$ is the volume of its truncation; in the same spirit, the \emph{length} of an edge of $P$ is the length of its subsegment contained in the truncation. Notice that the volume of a non-empty generalized hyperbolic polyhedron could be $0$ if the truncation is empty; likewise the length of some of its edges could be $0$.
 \end{itemize}

\end{dfn}

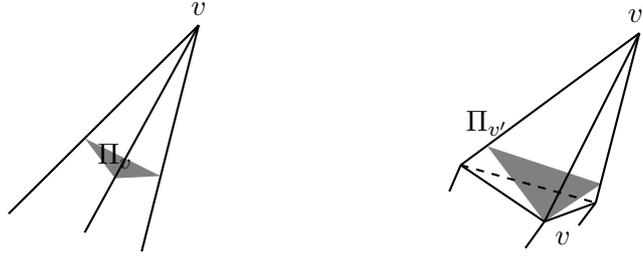
\begin{figure}
 \centering
 \begin{minipage}{.45\textwidth}
  \centering
    \begin{tikzpicture}
\centering
\fill[fill=gray](2.5,2.5)--(3.5,2)--(2.9,1.97)node[above]{$\Pi_v$};
\draw[thick] (1.5,1.5)--(4,4)node[above]{$v$};
\draw[thick](3.25,1)--(4,4);
\draw[thick] (2.5,1.25)--(4,4);
\end{tikzpicture}
 \end{minipage}
\begin{minipage}{.45\textwidth}
 \centering
   \begin{tikzpicture}
\centering

\fill[fill=gray](3,3.5)node[above]{$\Pi_{v'}$}--(4.5,3)--(3.75,2.5);
\draw[thick](5,5)--(4.425,2.75);
\draw[thick](4.2,2.45)--(4.425,2.75);
\draw [thick](5,5)node[above]{$v'$}--(2.65,3.25);
\draw[thick](2.65,3.25)--(2.5,2.9);
\draw[thick] (5,5)--(3.75,2.5);
\draw[thick] (3.5,2.15)--(3.75,2.5);
\draw[thick] (4.425,2.75)--(3.75,2.5)node[below right]{$v$};
\draw[thick](3.75,2.5)--(2.65,3.25);
\draw[thick,dashed](2.65,3.25)--(4.425,2.75);
\end{tikzpicture}
\end{minipage}
\caption{A proper (left) and almost proper (right) truncation.}\label{fig:propbadtrunc}
\end{figure}

In the remainder of the paper we simply say \emph{proper polyhedra} (or \emph{almost proper
polyhedra}) for proper (respectively, almost proper) generalized hyperbolic polyhedra. We are mostly interested in proper polyhedra; almost proper polyhedra can arise as limits of proper polyhedra, and they will be studied carefully in the proof of Theorem \ref{maxvol}.

When it has positive volume, the truncation of a generalized hyperbolic polyhedron $P$ is itself a polyhedron; some of its faces are the truncation of the faces of $P$, while the others are the intersection of $P$ with some truncating plane; we call such faces \emph{truncation faces}. Notice that distinct truncation faces are disjoint (even more, the planes containing them are disjoint) \cite[Lemma 4]{bonbao}. If an edge of the truncation of $P$ is not the intersection of an edge of $P$ with the truncating half-spaces, then we say that the edge is arising from the truncation. Every edge that arises from truncation is an edge of a truncation faces. The converse is true for proper polyhedra but not necessarily for almost proper ones: it could happen that an entire edge of $P$ lies in a truncation plane, and we do not consider this to be an edge arising from the truncation.

\begin{oss}
 For both almost proper and proper polyhedra the dihedral angles at the edges arising from the truncation are $\frac{\pi}{2}$.
\end{oss}

\begin{oss}
 An important feature of the truncation of a proper polyhedron $P$ is that it determines $P$ (once we know which faces of $P$ are the truncation faces), since it is enough to remove the truncation faces to undo the truncation (see Figure \ref{fig:recover}). This also holds for almost proper polyhedra (see Figure \ref{fig:improperdef} in Section \ref{sec:maxvol}). In particular this will allow us to use many standard techniques to study them, such as the Schl\"afli formula (see Theorem \ref{teo:schlafli}).
\end{oss}

  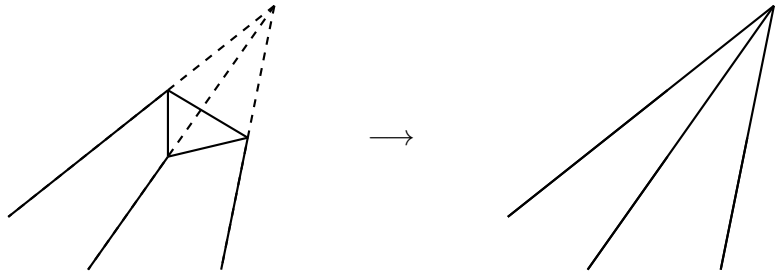
\begin{figure}
  \centering
\begin{minipage}{.45\textwidth}
\centering
 \begin{tikzpicture}[scale=0.7]
\draw [thick,dashed](5,5)--(4,0);
\draw [thick, dashed](5,5)--(1.5,0);
\draw[thick,dashed](5,5)--(0,1);
\draw [thick] (4,0)--(4.5,2.5);
\draw [thick] (0,1)--(3,3.4);
\draw [thick] (1.5,0)--(3,2.14);
\draw [thick] (4.5,2.5)--(3,3.4);
\draw [thick](3,2.14)- -(3,3.4);
\draw [thick] (3,2.14)--(4.5,2.5);
\end{tikzpicture}
\end{minipage}
$\longrightarrow$
\begin{minipage}{.45\textwidth}
\centering
 \begin{tikzpicture}[scale=0.7]
\draw [thick](5,5)--(4,0);
\draw [thick](5,5)--(1.5,0);
\draw[thick](5,5)--(0,1);
\draw [thick] (4,0)--(4.5,2.5);
\draw [thick] (0,1)--(3,3.4);
\draw [thick] (1.5,0)--(3,2.14);
\end{tikzpicture}
\end{minipage}
\caption{Removing the truncation faces recovers the original polyhedron.}\label{fig:recover}
 \end{figure}

We are always going to consider \emph{face marked} polyhedra; this means that each face of a polyhedron is uniquely determined, and therefore they never have any symmetry.

\begin{oss}\label{rem:unique}
 If $\Gamma$ is the $1$-skeleton of a projective polyhedron, then it is $3$-connected (that is to say, it cannot be disconnected by removing two vertices). Conversely, any $3$-connected planar graph is the $1$-skeleton of a proper polyhedron \cite{steinitz}. 
 If a planar graph is $3$-connected, then it admits a unique embedding in $S^2$ (up to isotopies of $S^2$ and mirror symmetry) \cite[Corollary 3.4]{fle}. Hence when in the following we consider a planar graph $\Gamma$, it is always going to be $3$-connected and embedded in $S^2$. In particular, it will make sense to talk about the faces of $\Gamma$ and the dual of $\Gamma$, denoted with $\Gamma^*$. The graph $\Gamma^*$ is the $1$-skeleton of the cellular decomposition of $S^2$ dual to that of $\Gamma$. Notice that if $\Gamma$ is the $1$-skeleton of a polyhedron $P$, then $\Gamma^*$ is the $1$-skeleton of the polyhedron whose vertices are dual to the faces of $P$, hence $\Gamma$ is $3$-connected if and only if $\Gamma^*$ is.
\end{oss}

\begin{dfn}
 Let $\Gamma$ be a planar $3$-connected graph; the space of all the face-marked proper polyhedra with $1$-skeleton $\Gamma$ considered up to isometry (i.e. projective transformations preserving the unit sphere) is denoted as $\mathcal{A}_\Gamma$.
\end{dfn}

\begin{oss}
 It is important not to mix up the $1$-skeleton of a projective polyhedron with the $1$-skeleton of its truncation. In what follows, whenever we refer to
 $1$-skeleta we always refer to those of projective polyhedra (and not their truncation) unless specified.
\end{oss}

Whether a vertex of a polyhedron is real, ideal or hyperideal can be read directly from the dihedral angles.

\begin{lem}\label{lem:angid}
 Let $P$ be a generalized hyperbolic polyhedron, $v$ a vertex of $P$ and $\theta_1,\dots,\theta_k$ the dihedral angles of the edges incident to $v$. Then $v$ is hyperideal if and only if $\theta_1,\dots,\theta_k$ are the angles of a hyperbolic $k$-gon; $v$ is ideal if and only if $\theta_1,\dots,\theta_k$ are the angles of a Euclidean $k$-gon; $v\in\h$ if and only if $\theta_1,\dots,\theta_k$ are the angles of a spherical $k$-gon. Equivalently, 
 \begin{itemize}
  \item $v$ is hyperideal if and only if $\sum_i \theta_i<(k-2)\pi;$
  \item $v$ is ideal if and only if $\sum_i \theta_i=(k-2)\pi;$
  \item $v\in\h$ if and only if $\sum_i\theta_i>(k-2)\pi$.
 \end{itemize}
\end{lem}

For a proof of this Lemma see for example \cite[Proposition 5]{bonbao}.

Finally in the proof of the main theorem we will need a way to deform a almost proper polyhedron to be proper. We will rely on the following easy lemma.

\begin{lem}\label{lem:unproper}
 Let $v\in\rp\backslash \overline{\h}$ and $w\in H_v$. If $\Psi$ is a translation of $\R^3$ or a homothety centered in $0$ such that $\Psi(v)$ is contained in the tangent cone of $v$ to $\partial\h$, then if $\Psi(w)\in\h$ it is also contained in the interior of $H_{\Psi(v)}$. 
\end{lem}
Notice that in particular this lemma says that if $w$ is an almost proper vertex of a polyhedron $P$, then $\Psi(w)$ is a proper vertex of $\Psi(P)$.
\begin{proof}
 The plane $\Pi_{\Psi(v)}$ is disjoint from $H_v$ (see Figure \ref{fig:unproper}), and certainly $\Psi(w)\in H_v$.
\end{proof}

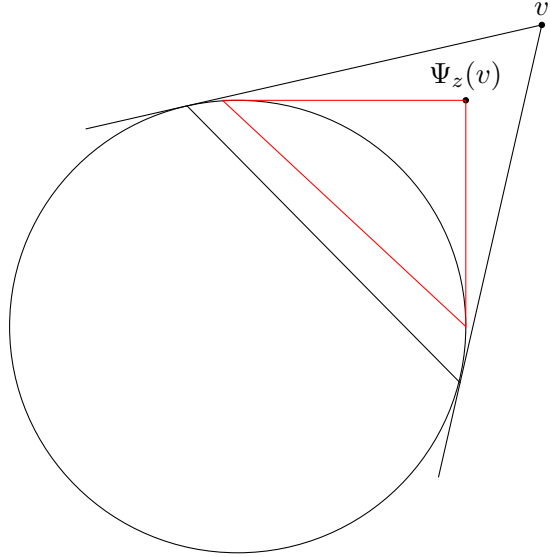
\begin{figure}
 \centering
  \begin{tikzpicture}
  \draw (1,1) circle[radius=3cm];
  \draw[fill=black] (5,5) node[above]{$v$} circle[radius=1pt];
  \draw[fill=black] (4,4) node[above]{$\Psi_z(v)$} circle[radius=1pt];
  \draw (5,5)--(-1,3.62);
  \draw (5,5)--(3.64,-1);
  \draw (0.33,3.93)--(3.93,0.25);
  \draw [color=red] (4,4)--(0.8,4);
  \draw [color=red] (4,4)--(4,1)--(0.8,4);
 \end{tikzpicture}
\caption{Pushing $P$ towards $\h$ pushes its dual plane away from the center}\label{fig:unproper}
\end{figure}

\section{The space of proper polyhedra and the volume function}\label{sec:background}

\subsection{The Bao-Bonahon existence and uniqueness theorem for hyperideal polyhedra}

A special class of proper polyhedra is that of the \emph{hyperideal polyhedra}, i.e. generalized hyperbolic polyhedra with no real vertices. Since there are no real vertices, hyperideal polyhedra are automatically proper. In \cite{bonbao}, Bao and Bonahon gave a complete description of the space of angles of hyperideal polyhedra.

\begin{teo}\label{teo:bonbao}
 Let $\Gamma$ be a $3$-connected planar graph with edges $e_1,\dots,e_k$. There exists a hyperideal polyhedron $P$ with $1$-skeleton $\Gamma$ and dihedral angles $\theta_1,\dots,\theta_k\in(0,\pi)$ at the edges $e_1,\dots,e_k$ if and only if the following conditions are satisfied:
 \begin{itemize}
  \item for any closed curve $\gamma\subseteq S^2$ passing transversely through distinct edges $e_{i_1},\dots,e_{i_h}$ of $\Gamma$ exactly once, we have $\sum_{j=1}^h\theta_{i_j}\leq (h-2)\pi$ with equality possible only if $e_{i_1},\dots,e_{i_h}$ share a vertex;
  \item for any arc $\gamma\subseteq S^2$ with endpoints in two different faces sharing a vertex, and passing transversely through the distinct edges $e_{i_1},\dots,e_{i_h}$ of $\Gamma$ exactly once, we have $\sum_{j=1}^h\theta_{i_j}< (h-1)\pi$ unless the edges $e_{i_1},\dots,e_{i_h}$ share a vertex.
 \end{itemize}
Moreover, if $P$ exists it is unique up to isometry.
\end{teo}

In particular, this theorem says that dihedral angles uniquely determine a hyperideal polyhedron, and the space of angles of hyperideal polyhedra with fixed $1$-skeleton is a convex subset of $\R^k$.

 \subsection{The space of proper polyhedra \texorpdfstring{$\mathcal{A}_\Gamma$}{Ag}}\label{sec:polspace}
 
 Let $\Gamma$ be a $3$-connected planar graph, and denote with $\mathcal{P}_\Gamma$ the set of face-marked proper polyhedra with $1$-skeleton $\Gamma$; denote with $\mathcal{A}_\Gamma$ the set of isometry classes in $\mathcal{P}_\Gamma$.
 
The first result about $\mathcal{A}_\Gamma$ that we need is an explicit description of the set $\mathcal{A}_\Gamma$.

\begin{prop}\label{prop:manifold}
 The set $\mathcal{A}_\Gamma$ is naturally a smooth manifold.
\end{prop}
\begin{proof}
Montcouquiol proved in \cite{mont} that face-marked Euclidean polyhedra in $\R^3$ with $1$-skeleton $\Gamma$ form a smooth submanifold of $((\mathbb{RP}^3)^*)^F$, with $F$ the number of faces of $\Gamma$. Since proper polyhedra are in natural 1-1 correspondence with an open subset of Euclidean polyhedra, they are also a smooth submanifold of $((\mathbb{RP}^3)^*)^F$. To conclude we notice that the action of the isometries of $\h$ on this space of polyhedra is free and proper, so that the quotient $\mathcal{A}_\Gamma$ is a manifold as well. 
\end{proof}

 Consider the dihedral angle map $\Theta:\mathcal{A}_\Gamma\ra \R^{\# \textrm{ of edges}}$ assigning to each polyhedron the tuple of dihedral angles of its edges. This is clearly a smooth map; it is also a local diffeomorphism at any polyhedron with no ideal vertices.

 \begin{teo}\cite[Theorem 19]{mont},\cite[Theorem 1.1]{weiss}\label{teo:loccoord}
  If $P$ is a compact hyperbolic polyhedron (i.e. $P$ has only real vertices), the dihedral angles are local coordinates for $\mathcal{A}_\Gamma$ around $P$.
 \end{teo}
 
 \begin{cor}\label{cor:loccoord}
  If $P$ is a proper polyhedron with no ideal vertices, the dihedral angles are local coordinates for $\mathcal{A}_\Gamma$ around $P$.
 \end{cor}

\begin{proof}
 Take $P^0$ the truncation of $P$, and denote with $\Gamma^0$ its $1$-skeleton. Since $P^0$ is compact the dihedral angles are local coordinates for $\mathcal{A}_{\Gamma^0}$ around $P^0$; the dihedral angles of $P^0$ are either $\frac{\pi}{2}$ (at the edges lying on truncation faces) or those of $P$ (at the remaining edges). Then the dihedral angles of $P$ form a local set of coordinates for the subset of $\mathcal{A}_{\Gamma^0}$ of polyhedra with right angles at the edges arising from truncation, and any polyhedron in this subset is going to be the truncation of a proper polyhedron in $\mathcal{A}_\Gamma$ close to $P$.
\end{proof}

\begin{dfn}
 The \emph{closure} of $\mathcal{A}_\Gamma$, denoted with $\overline{\mathcal{A}_\Gamma}$, is the topological closure of the space of proper polyhedra with $1$-skeleton $\Gamma$ (as a subset of $((\mathbb{RP}^3)^*)^F$), quotiented by isometries. As customary the \emph{boundary} of $\mathcal{A}_\Gamma$ is $\overline{\mathcal{A}_\Gamma}\backslash \mathcal{A}_\Gamma$ and is denoted with $\partial \mathcal{A}_\Gamma$. We make no claim that $\overline{\mathcal{A}_\Gamma}$ is a manifold; even if it were, its boundary as a manifold would not necessarily be $\partial\mathcal{A}_\Gamma$.
 We say that a sequence $P_n\in\mathcal{A}_\Gamma$ converges to $P\in\overline{\mathcal{A}_\Gamma}$ if it converges in the topology of $\overline{\mathcal{A}_\Gamma}$.
\end{dfn}

We will also need local coordinates for certain parts of $\partial{\mathcal{A}_\Gamma}$; this is provided by the following corollary.

\begin{cor}\label{cor:locbad}
If $P$ is an almost proper polyhedron with $1$-skeleton $\Gamma$ and no ideal vertices, it lies in the closure of $\mathcal{A}_\Gamma$. Moreover, if $\vec{\theta}=(\theta_1,\dots,\theta_e)$ are the dihedral angles of the proper edges of $P$, and $\overrightarrow{\theta'}=(\theta_1',\dots,\theta_e')$ is close enough to $\vec{\theta}$, then there is a unique almost proper polyhedron with $1$-skeleton $\Gamma$ close to $P$ in $\overline{\mathcal{A}_\Gamma}$ with the same almost proper vertices and with angles $\overrightarrow{\theta'}$.
\end{cor}

\begin{proof}
 To show the fact that $P\in\overline{\mathcal{A}_\Gamma}$ we need to exhibit a family of proper polyhedra with $1$-skeleton $\Gamma$ converging to $P$ in $\left({(\rp)}^*\right)^F$. To do this apply an isometry so that $0\in P$, and consider $\Phi_\lambda:\R^3\ra \R^3$ the multiplication by $\lambda$. Then for $\lambda\in(1-\epsilon,1]$ the polyhedron $\Phi_\lambda(P)$ is proper by Lemma \ref{lem:unproper}, has $1$-skeleton $\Gamma$ and converges to $P$ as $\lambda\ra1$. To see that $\Phi_\lambda(P)$ is proper notice that for every hyperideal vertex $v$, $\Phi_\lambda(v)$ is contained in the tangent cone of $v$ to $\partial\h$, and we can conclude by applying Lemma \ref{lem:unproper}.
 
 To show the second assertion, reorder the indices so that $\theta_1,\dots,\theta_l$ are the angles of the proper edges that are contained in some truncation plane. Then $P_0$, the truncation of $P$, is compact, has $1$-skeleton $\Gamma_0$ and dihedral angles $\theta_1-\frac{\pi}{2},\dots,\theta_l-\frac{\pi}{2},\theta_{l+1},\dots,\theta_e,\frac{\pi}{2},\dots,\frac{\pi}{2}$. By Theorem \ref{teo:loccoord}, if $\theta_1',\dots,\theta'_e$ are sufficiently close to $\theta_1,\dots,\theta_e$ there is a unique $Q_0$ (up to isometry) close to $P_0$ with $1$-skeleton $\Gamma_0$ and angles $\theta'_1-\frac{\pi}{2},\dots,\theta'_l-\frac{\pi}{2},\theta'_{l+1},\dots,\theta'_e,\frac{\pi}{2},\dots,\frac{\pi}{2}.$ Some faces of $Q_0$ correspond to truncation faces of $P_0$; if we glue to $Q_0$ the convex hull of a truncation face and its dual point, we undo the truncation (see Figure \ref{fig:improperdef}). Notice that the angles at the edges that are glued are either $\theta'_i-\frac{\pi}{2}+\frac{\pi}{2}=\theta'_i$ if $i\leq l$, or $\frac{\pi}{2}+\frac{\pi}{2}$ otherwise (hence the edge in this case disappears).
 
 If we undo every truncation in this manner, we obtain an almost proper polyhedron $Q$ with $1$-skeleton $\Gamma$ and angles $\theta'_1,\dots,\theta'_e$. To see that $Q$ is close to $P$ notice that $Q_0$ is close to $P_0$, which means that all the faces of $Q_0$ are close to the corresponding faces of $P_0$; but the faces of $P$ and $Q$ depend continuously on the faces of $P_0$ and $Q_0$.
 
 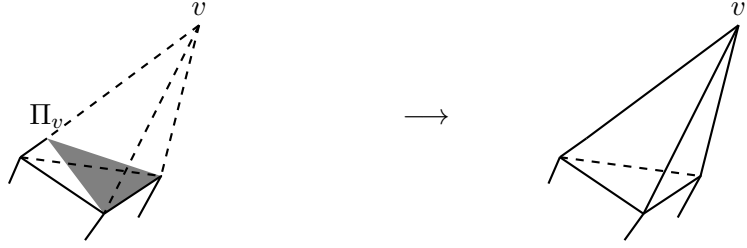
\begin{figure}
 \centering
 \begin{minipage}{.4\textwidth}
     \begin{tikzpicture}
\centering
\fill[fill=gray](3,3.5)node[above]{$\Pi_{v}$}--(4.5,3)--(3.75,2.5);
\draw[thick,dashed](5,5)node[above]{$v$}--(4.5,3);
\draw[thick](4.2,2.45)--(4.5,3);
\draw [thick,dashed](5,5)--(3,3.5);
\draw[thick](2.65,3.25)--(2.5,2.9);
\draw[thick](3,3.5)--(2.65,3.25);
\draw[thick](4.5,3)--(4.5,3);
\draw[thick,dashed] (5,5)--(3.75,2.5);
\draw[thick] (3.5,2.15)--(3.75,2.5);
\draw[thick] (4.5,3)--(3.75,2.5);
\draw[thick](3.75,2.5)--(2.65,3.25);
\draw[thick,dashed](2.65,3.25)--(4.5,3);
\end{tikzpicture}
 \end{minipage}
$\longrightarrow$
 \begin{minipage}{.4\textwidth}
  \centering
     \begin{tikzpicture}
\centering
\draw[thick](5,5)node[above]{$v$}--(4.5,3);
\draw[thick](4.2,2.45)--(4.5,3);
\draw [thick](5,5)--(3,3.5);
\draw[thick](2.65,3.25)--(2.5,2.9);
\draw[thick](3,3.5)--(2.65,3.25);
\draw[thick](4.5,3)--(4.5,3);
\draw[thick] (5,5)--(3.75,2.5);
\draw[thick] (3.5,2.15)--(3.75,2.5);
\draw[thick] (4.5,3)--(3.75,2.5);
\draw[thick](3.75,2.5)--(2.65,3.25);
\draw[thick,dashed](2.65,3.25)--(4.5,3);
\end{tikzpicture}
 \end{minipage}
 \caption{Removing a truncation face to recover an almost proper polyhedron}\label{fig:improperdef}
\end{figure}
 \end{proof}

 \subsection{Convergence of polyhedra}
 
 As we have seen in Subsection \ref{sec:polspace}, a proper polyhedron $P$ is naturally an element of $\left((\rp)^*\right)^F$ where $F$ is the number of faces of $P$. Therefore when we say that a sequence $P_n$ of polyhedra with $1$-skeleton $\Gamma$ converges to $P\in\left((\rp)^*\right)^F$ (or has $P$ as an accumulation point) we mean in the topology of $\left((\rp)^*\right)^F$. Since $\rp$ is compact, any such sequence $P_n$ has an accumulation point $P\in\left((\rp)^*\right)^F$. Each component of $P_n$ is a plane in $\rp$, and $P_n$ selects one of the two hyperspaces on either side of it. Therefore if $P_n\ra P$ then $P$ can still be interpreted as a convex intersection of half-spaces, however it could stop being a projective polyhedron if it is contained in a plane or not contained in an affine chart. 
 
 Analogously, we will consider sequences of $k$-gons $A_n\subseteq\mathbb{H}^2$, which are naturally elements of $\left((\mathbb{RP}^2)^*\right)^k$; when we say that they converge to a polygon $A$, we mean in the topology of this space, with the same considerations we made for polyhedra.
 
 If the limit point $P$ is a projective polyhedron, each of its vertices is a limit of some vertices of $P_n$, similarly every line containing an edge of $P$ is the limit of some lines containing edges of $P_n$ and every plane containing a face of $P$ is the limit of some planes containing a face of $P_n$ respectively. However some vertices of $P_n$ could converge to points of $P$ which are not vertices; similarly some edge of $P_n$ could collapse to a point or converge to a segment which is not an edge, and a face of $P_n$ could collapse to an edge or a point.
 
 Furthermore in some degenerate cases, even though $P_n$ converges to $P$, some of its vertices do not converge; for example, if $v_n$ lies on three faces of $P_n$ that become coplanar, the sequence $v_n$ could have accumulation points anywhere on the limit face. Throughout the paper we are almost always going to be concerned with sequences of polyhedra with angles which are decreasing and bounded away from $0$; in this case the situation is somewhat nicer.
 
 \textbf{Notation:} throughout the rest of the paper we consider sequences of polyhedra $P_n$ with the same $1$-skeleton $\Gamma$. Their boundary will always be equipped with a fixed isomorphism to the pair $(S^2,\Gamma)$; because of Remark \ref{rem:unique}, there is essentially a unique way to do this, so it will be not explicitly defined. If we consider a sequence of vertices $v_n\in P_n$ we always assume that each $v_n$ is the vertex of $P_n$ corresponding to a fixed vertex $v$ of $\Gamma$; the same with a sequence of edges or faces.
 
 \begin{lem}\label{lem:coll1}
  Let $P_n$ be a sequence of proper or almost proper polyhedra with $1$-skeleton $\Gamma$ and with angles bounded away from $0$ and $\pi$, and converging to the projective polyhedron $P$. If $\Pi^1_n\neq \Pi^2_n$ are planes containing faces of $P_n$ converging to the planes $\Pi^1,\Pi^2$ containing faces $F^1,F^2$ of $P$, then $\Pi^1\neq \Pi^2$.
 \end{lem}
\begin{proof}
Suppose by contradiction that $\Pi^1=\Pi^2$. If $\Pi_1^n$ and $\Pi_2^n$ share an edge of $P_n$ this would imply that the dihedral angle at this edge would either converge to $\pi$ or $0$ which is a contradiction.
If instead $\Pi_1^n$ and $\Pi_2^n$ do not share an edge in $P_n$, still there must be some other $\Pi_3^n$ containing a face of $P_n$ such that $\Pi_1^n$ and $\Pi_3^n$ must share an edge and $\Pi_3^n\ra \Pi_1$; otherwise, if every plane containing a face adjacent to $\Pi_1^n$ converged to a plane different from $\Pi_1$, then $P$ would not be convex.
\end{proof}

\begin{cor}\label{cor:edgeconv}
 With the same hypotheses of Lemma \ref{lem:coll1}, if $v_n$ is a sequence of vertices of $P_n$, then it cannot converge to (or have an accumulation point in) an internal point of a face of $P$. Similarly, any accumulation point of a sequence of edges $e_n$ of $P_n$ cannot intersect the interior of any face of $P$.
\end{cor}
\begin{proof}
 If $v_n$ converged to (or had an accumulation point in) an internal point of a face of $P$, then all faces of $P_n$ containing $v_n$ would become coplanar in the limit, contradicting Lemma \ref{lem:coll1}.
 
 Consider now a sequence of edges $e_n$: following the same reasoning, if $e_n$ converged (or has an accumulation point) to an edge that intersected the interior of a face of $P$, the two faces of $P_n$ containing $e_n$ would become coplanar in the limit.
\end{proof}

\begin{cor}
If $v_n$ converges to an internal point of an edge, then the accumulation points of all edges that have $v_n$ as endpoints must be contained in that edge as well.\end{cor}

 Suppose that $P_n$ satisfies all the hypotheses of Lemma \ref{lem:coll1}, that $P_n\ra P$ with $P$ a generalized hyperbolic polyhedron with $1$-skeleton $\Gamma'$ and additionally that all vertices of $P_n$ converge. Then the limit induces a simplicial map $\phi:\Gamma\ra\tilde{\Gamma'}$ where $\tilde{\Gamma'}$ is obtained from $\Gamma'$ by adding bivalent vertices to some of its edges. The map $\phi$ sends
 \begin{itemize}
  \item  each vertex of $\Gamma$ to its limit in $\tilde{\Gamma'}$;
  \item each edge of $\Gamma$ linearly to the convex hull of the image of its endpoints (the convex hull is contained in $\tilde{\Gamma'}$ by Corollary \ref{cor:edgeconv}).
 \end{itemize}

 Notice that $\tilde{\Gamma'}$ is not always $3$-connected, however it is still $2$-connected since it cannot be disconnected by removing a single vertex.

 \begin{figure}
  \centering
  \begin{minipage}{.45\textwidth}
   \centering
    \begin{tikzpicture}
 \draw[thick] (0,2)--(1,1)--(0,0);
 \draw[thick](0,1)--(4,1);
 \draw[thick] (4,2)--(3,1)--(4,0);
 \end{tikzpicture}
  \end{minipage}
$\longrightarrow$
  \begin{minipage}{.45\textwidth}
   \centering
    \begin{tikzpicture}
 \draw[thick] (0,2)--(1,1)--(0,0);
 \draw[thick](0,1)--(2,1);
 \draw[thick] (2,2)--(1,1)--(2,0);
 \end{tikzpicture}
  \end{minipage}
  \caption{An edge collapsing to a vertex.}\label{fig:edgecollapse}
 \end{figure}
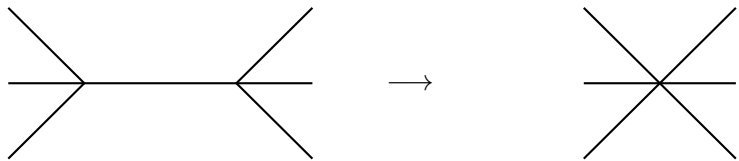

 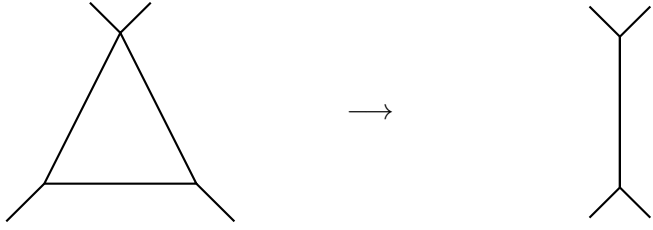
\begin{figure}
  \centering
  \begin{minipage}{.45\textwidth}
   \centering
 \begin{tikzpicture}
 \draw[thick] (2,2)--(3,0)--(1,0)--(2,2);
 \draw[thick] (1.6,2.4)--(2,2)--(2.4,2.4);
 \draw[thick] (1,0)--(0.5,-0.5);
 \draw[thick] (3,0)--(3.5,-.5);
 \end{tikzpicture}
  \end{minipage}
$\longrightarrow$
  \begin{minipage}{.45\textwidth}
   \centering
 \begin{tikzpicture}
 \draw[thick] (2,2)--(2,0)--(2,0)--(2,2);
 \draw[thick] (1.6,2.4)--(2,2)--(2.4,2.4);
 \draw[thick] (2,0)--(1.6,-0.4);
 \draw[thick] (2,0)--(2.4,-.4);
 \end{tikzpicture}
  \end{minipage}
  \caption{A face collapsing to an edge.}\label{fig:facecollapse}
 \end{figure}
 \begin{lem}\label{lem:coll2}
  The map $\phi:\Gamma\ra\tilde{\Gamma'}$ factors through a graph $\hat{\Gamma}$, isomorphic to $\tilde{\Gamma'}$ and obtained from $\Gamma$ via a finite sequence of the following moves:
   \begin{enumerate}[label=(\roman*)]
    \item an edge of $\Gamma$ collapses to a vertex (see Figure \ref{fig:edgecollapse});
    \item a face of $\Gamma$ collapses to an edge (see Figure \ref{fig:facecollapse}).
   \end{enumerate}
 \end{lem}
\begin{proof}

Let $\Lambda_v\subseteq \Gamma$ be the subset $\Lambda_v:=\phi^{-1}(v)$ for any $v$ vertex of $\tilde{\Gamma'}$.

The map $\phi$ satisfies the following properties:

\begin{itemize}
 \item $\phi$ is a surjective simplicial map between $2$-connected planar graphs;
 \item for every $v$ vertex of $\tilde{\Gamma'}$, the set $\Lambda_v$ is connected;
 \item if a cycle $\lambda$ is contained in $\Lambda_v$, then all vertices in one of the two components of the sphere bounded by $\lambda$ must be in $\Lambda_v$.
\end{itemize}

The fact that $\phi$ is surjective, simplicial and between $2$-connected planar graphs is obvious; let us prove that $\Lambda_v$ is connected. 

First notice that if $w_1,w_2\in\Lambda_v$ lie on the same face of $\Gamma$, they are in the same connected component of $\Lambda_v$ (even if they are not vertices of $\Gamma$). This is because the face must correspond to a convex face of $P_n$, and if two different vertices of a sequence of convex polygons coincide in the limit, then by convexity one of the two paths connecting them (on either side of the boundary of the polygon) must coincide in the limit as well. 

Take now generic $w_1,w_2\in\Lambda_v$: they determine points $w_1^n,w_2^n\in P_n$. Pick any converging sequence of planes $\Pi_n$ containing $w^n_1,w^n_2$ and cut $P_n$ along $\Pi_n$; for $n$ big enough the planes can be chosen so that the $1$-skeleta of the resulting polyhedra do not change with $n$ (since they depend on the edges of $P_n$ that are intersected by $\Pi_n$). This gives a sequence of polyhedra $R_n$ with $w_1^n$ and $w_2^n$ collapsing to the same point as $n\ra\infty$; moreover now $w^n_1$ and $w^n_2$ share a face $F_n\subseteq \Pi_n$. Then there is a sequence $z^n_1,\dots,z^n_k$ of vertices of $R_n$ that all converge to $v$; each $z^n_i$ shares an edge with $z^n_{i-1}$ and $z^n_{i+1}$. The points $z^n_1,\dots,z^n_k$ can also be naturally viewed as points in the $1$-skeleton of $P_n$: notice that $z_i$ shares a face with $z_{i+1}$ in $\Gamma$, therefore they are all in the same connected component of $\Lambda_v$. Similarly $w_1$ lies on the same face of $z_1$ and $w_2$ lies on the same face of $z_k$, and we obtain that $w_1$ and $w_2$ are in the same connected component of $\Lambda_v$.

Finally, let $\lambda$ be a cycle contained in $\Lambda_v$, and $K_1,K_2$ the two connected components of $S^2$ bounded by $\lambda$. If $v_1\in \Gamma\cap K_1$ and $v_2\in\Gamma\cap K_2$ and neither of them was in $\Lambda_v$, then $\tilde{\Gamma'}$ could be disconnected by removing $v$. To see this, notice that $\phi(v_1)\neq\phi(v_2)$ and take any path $\gamma$ connecting $\phi(v_1)$ to $\phi(v_2)$. We wish to prove that $\gamma$ must contain $v$, and to do so we show that there is a path from $v_1$ to $v_2$ contained in $\phi^{-1}(\gamma)$ (such a path has to cross $\lambda\subseteq\Lambda_v$, hence $v\in\gamma$). Take the edge $e$ of $\gamma$ containing $\phi(v_1)$, and lift it to any edge $\tilde{e}$ of $\Gamma$. If $v_1\notin \tilde{e}$, it can connected to one of its endpoint via a path contained in $\Lambda_{\phi(v_1)}$ since this is connected. We can keep lifting the path $\gamma$ to a path contained in $\phi^{-1}(\gamma)$, one edge at a time, gluing endpoints with a path contained in $\phi^{-1}(\gamma)$ when necessary.

To see this, notice that $\phi(v_1)\neq\phi(v_2)$ and any path $\gamma$ connecting $\phi(v_1)$ to $\phi(v_2)$ must pass through $v$, since it would lift to a path contained in $\phi^{-1}(\gamma)$ connecting $v_1$ to $v_2$ which has to pass through $\lambda$. This contradicts the fact that $\tilde{\Gamma'}$ is $2$-connected.

We now prove that any map $\phi:\Gamma_1\ra\Gamma_2$ satisfying the previous $3$ properties must factor through a graph $\tilde{\Gamma}$ obtained via edge or face collapses as in the thesis of the Lemma.
 Let $n$ be the number of vertices of $\Gamma_1$ and $m$ the number of vertices of $\Gamma_2$, and proceed by induction on $n-m$.
 
 If $n-m=0$ then $\Gamma_1$ is isomorphic to $\Gamma_2$ and there is nothing to prove.
 
 For the inductive step, we have $n>m$ which implies there is a $w$ such that $\Lambda_w$ contains more than one vertex; we have by hypothesis that $\Lambda_w$ is a connected subgraph of $\Gamma_1$. 
 
 If $\Lambda_w$ has a leaf, then $\phi$ factors through the graph obtained from $\Gamma_1$ by contracting the two vertices of this leaf. If $\Lambda_w$ contains a cycle $\lambda$, then all vertices in one of the two components of the plane bounded by $\lambda$ must be in $\Lambda_w$. In particular $\Lambda_w$ contains a cycle that bounds a face, and $\phi$ factors through the graph obtained from collapsing this face to an edge.
 
 In either case, the new graph has fewer vertices than $\Gamma_1$ and the induced $\phi$ has the same $3$ properties. We can then conclude by induction.
\end{proof}

\subsection{The rectification of a polyhedron}\label{sec:rect}

\begin{dfn}\label{dfn:rect}
We say that a projective polyhedron $\overline{\Gamma}$ is a \emph{rectification} of a $3$-connected planar graph $\Gamma$ if the $1$-skeleton of $\overline{\Gamma}$ is equal to $\Gamma$ and all the edges of $\overline{\Gamma}$ are tangent to $\partial\h$.
\end{dfn}

\begin{oss}
 Notice that, by definition, $\overline{\Gamma}$ is not a generalized hyperbolic polyhedron, since none of its edges intersect $\h$. However as we will see it is still possible to give a definition of the volume of $\overline{\Gamma}$ as for any proper polyhedron.
\end{oss}

\begin{oss}
 A rectification of $\Gamma$ gives a circle packing in $\partial\h=S^2$ with tangency graph $\Gamma^*$. To see this, take the circles arising from the intersection of $\overline{\Gamma}$ with $\partial \h$; it is immediate to see that they are a circle packing and that their tangency graph is $\Gamma^*$, since each circle corresponds to a face of $\overline{\Gamma}$ and two circles are tangent if and only if their faces share an edge. 
 
 From this we could quickly prove the existence and uniqueness of the rectification by invoking the Koebe-Thurston Theorem \cite[Corollary 13.6.2]{thurston} about the existence and uniqueness of circle packings; however Thurston's proof of this theorem requires implicitly the existence and uniqueness of the rectification. Therefore, we are going to prove this separately in Proposition \ref{prop:exuniq}, with essentially the same proof given in \cite{thurston}.
\end{oss}

\begin{oss}\label{rmk:tang}
If two planes of $\rp$ intersect in a line tangent to $\partial\h$, then their hyperbolic dihedral angle is $0$; furthermore two distinct planes which intersect in $\overline{\h}$ have dihedral angle $0$ if and only if they intersect in a line tangent to $\partial \h$. Therefore, a projective polyhedron is a rectification of $\Gamma$ if and only if its $1$-skeleton is $\Gamma$, all its edges intersect $\overline{\h}$ and all its dihedral angles are $0$.
 
\end{oss}

The polyhedron $\overline{\Gamma}$ is not a generalized hyperbolic polyhedron, since its edges do not intersect $\h$. However its truncation can be defined in the same way as before, and it can be described very explicitly. Consider two vertices $v,v'$ of $\overline{\Gamma}$ connected by an edge $\overline{vv'}$ tangent to $\partial\h$. The planes $\Pi_{v}$ and $\Pi_{v'}$ intersect at the point of tangency, and the truncation of $\overline{\Gamma}$ is going to have an ideal, $4$-valent vertex with only right angles at that point. This can be repeated for every edge of $\overline{\Gamma}$ to see that its truncation has only right angles, and only ideal $4$-valent vertices. Some of its faces come from faces of $\Gamma$, while the others from its vertices. Notice that because of this, even though $\overline{\Gamma}$ is not a generalized hyperbolic polyhedron, its truncation is an ideal finite volume right-angled hyperbolic polyhedron and therefore $\overline{\Gamma}$ has a well defined hyperbolic volume.

This explicit description of the truncation of $\overline{\Gamma}$ quickly leads to the existence and uniqueness of the rectification.

\begin{prop}\label{prop:exuniq}
 For any $3$-connected planar graph $\Gamma$, the rectification $\overline{\Gamma}$ exists and is unique up to isometry.
\end{prop}
\begin{proof}
 Consider the planar graph obtained by taking a vertex for each edge of $\Gamma$, where two vertices are joined by an edge if and only if the corresponding edges share an angle (i.e. they share a vertex and they lie on the same face) in $\Gamma$. This graph is $4-$valent, and by Theorem \ref{teo:bonbao} it is the $1$-skeleton of a unique (up to isometry) right-angled ideal polyhedron $P$; this is going to be the truncation of $\overline{\Gamma}$. Some faces of $P$ correspond to vertices of $\Gamma$, while the others correspond to its faces. Then, the faces of $P$ corresponding to faces of $\Gamma$ bound the rectification $\overline{\Gamma}$ of $\Gamma$, and we can reconstruct uniquely $\overline{\Gamma}$ from $\Gamma$. 
\end{proof}
 
 \begin{oss}\label{oss:dual}
 Notice that from the polyhedron $P$ constructed in the proof of Proposition \ref{prop:exuniq} it is immediate to see that the rectification of $\Gamma$ and of its dual $\Gamma^*$ have the same truncation (and in particular the same volume): if we took the faces of $P$ corresponding to vertices of $\Gamma$, these would bound the rectification of $\Gamma^*$.
 \end{oss}
 
 \begin{ex}
  The truncation of the rectification of the tetrahedral graph (Figure \ref{fig:retttet}) is the right-angled hyperbolic ideal octahedron, whose volume is $v_8\cong 3.66$. As we noted in the introduction, it is proven in \cite[Theorem 4.2]{ushi} that indeed the supremum of volumes of proper hyperbolic tetrahedra is equal to $v_8$.
 \end{ex}

 \begin{figure}
  \centering
  \begin{minipage}{.4\textwidth}
   \centering
  \begin{tikzpicture}[rotate=20]
\centering
\draw (0,0) circle[radius=2cm];
\draw[thick] (2.2,2.2)--(1.84,-1.84)--(-2.2,-2.2)--(-1.84,1.84)--(2.2,2.2);
\draw[thick] (2.2,2.2)--(-2.2,-2.2);
\draw[thick,dashed](1.84,-1.84)--(-1.84,1.84);

\end{tikzpicture}
  \end{minipage}
  \hspace{1cm}
\begin{minipage}{.4\textwidth}
 \centering
   \begin{tikzpicture}[rotate=20]
\centering
\fill[color=gray](1.98,-0.15)--(-0.15,1.98)--(-0.1,-0.1)--(1.98,-0.15);
\fill[color=gray](-1.98,0.15)--(0.15,-1.98)--(-0.1,-0.1)--(-1.98,0.15);
\draw (0,0) circle[radius=2cm];
\draw[thick] (2.2,2.2)--(1.84,-1.84)--(-2.2,-2.2)--(-1.84,1.84)--(2.2,2.2);
\draw[thick] (2.2,2.2)--(-2.2,-2.2);
\draw[thick,dashed](1.84,-1.84)--(-1.84,1.84);
\draw[thin](1.98,-0.15)--(-0.15,1.98)--(-0.1,-0.1)--(1.98,-0.15);
\draw[thin](-1.98,0.15)--(0.15,-1.98)--(-0.1,-0.1)--(-1.98,0.15);
\draw[thin](-1.98,0.15)--(-0.15,1.98);
\draw[thin](1.98,-0.15)--(0.15,-1.98);
\draw[thin,dashed] (1.98,-0.15)--(0,0)--(0.15,-1.98);
\draw[thin,dashed] (-1.98,0.15)--(0,0)--(-0.15,1.98);
\end{tikzpicture}
\end{minipage}
\caption{The rectification of a tetrahedron (left) and its truncation (right), the ideal right-angled octahedron. The gray faces arise from the truncation of the top and bottom vertices.}\label{fig:retttet}
 \end{figure}
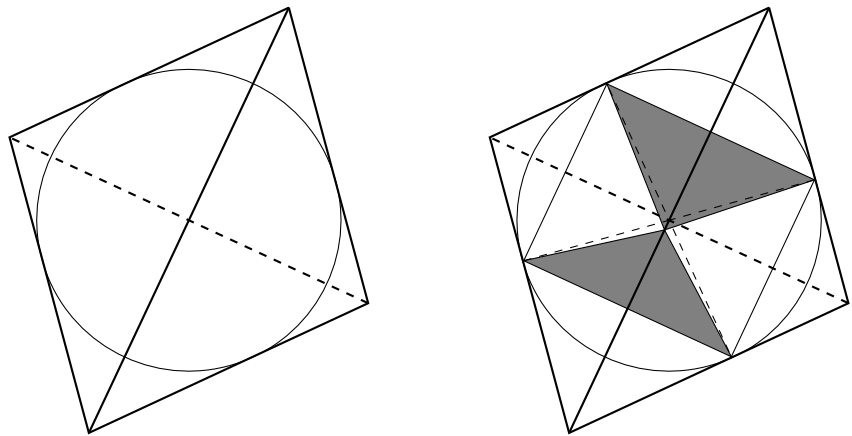
 
 \begin{ex}
  The truncation of the rectification of the $n$-gonal pyramid is the $n$-gonal antiprism. Its volume is given by the formula \cite{thurston}:
  
  \begin{displaymath}
   2n\left(\Lambda\left(\frac{\pi}{4}+\frac{\pi}{2n}\right)+\Lambda\left(\frac{\pi}{4}-\frac{\pi}{2n}\right)\right).
  \end{displaymath}

  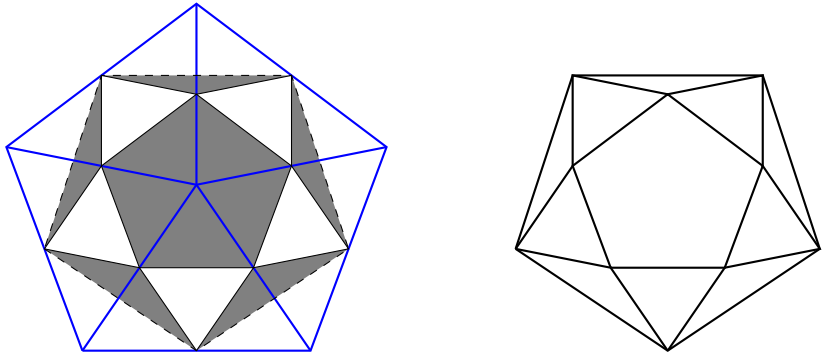
\begin{figure}
   \centering
      \begin{minipage}{.48\textwidth}
    \centering
 \begin{tikzpicture}
 
 \fill[color=gray] (0.75,1.1)--(2.25,1.1)--(2.75,2.45)--(1.5,3.4)--(0.25,2.45)--(0.75,1.1);
 \fill[color=gray] (1.5,0)--(2.25,1.1)--(3.5,1.35);
 \fill[color=gray] (3.5,1.35)--(2.75,2.45)--(2.75,3.65);
 \fill[color=gray] (2.75,3.65)--(1.5,3.4)--(0.25,3.65);
 \fill[color=gray] (0.25,3.65)--(0.25,2.45)--(-0.5,1.35);
 \fill[color=gray] (1.5,0)--(0.75,1.1)--(-0.5,1.35);
 \draw[thick,blue] (0,0)--(3,0)--(4,2.7)--(1.5,4.6)--(-1,2.7)--(0,0);
 \draw[thick,blue] (1.5,2.2)--(1.5,4.6);
 \draw[thick,blue] (1.5,2.2)--(0,0);
 \draw[thick,blue] (1.5,2.2)--(3,0);
 \draw[thick,blue] (1.5,2.2)--(4,2.7);
 \draw[thick,blue] (1.5,2.2)--(-1,2.7);
 \draw[thin] (1.5,0)--(0.75,1.1)--(-0.5,1.35);
 \draw[thin, dashed] (-0.5,1.35)--(1.5,0)--(3.5,1.35)--(2.75,3.65)--(0.25,3.65)--(-0.5,1.35);
 \draw[thin] (0.75,1.1)--(2.25,1.1)--(2.75,2.45)--(1.5,3.4)--(0.25,2.45)--(0.75,1.1);
 \draw[thin] (1.5,0)--(2.25,1.1)--(3.5,1.35)--(2.75,2.45)--(2.75,3.65)--(1.5,3.4)--(0.25,3.65)--(0.25,2.45)--(-0.5,1.35);
  \end{tikzpicture}
   \end{minipage}
   \begin{minipage}{.48\textwidth}
    \centering
     \begin{tikzpicture}
  \draw[thick,white] (0,0)--(3,0)--(4,2.7)--(1.5,4.6)--(-1,2.7)--(0,0);
 \draw[thick,white] (1.5,2.2)--(1.5,4.6);
 \draw[thick,white] (1.5,2.2)--(0,0);
 \draw[thick,white] (1.5,2.2)--(3,0);
 \draw[thick,white] (1.5,2.2)--(4,2.7);
 \draw[thick,white] (1.5,2.2)--(-1,2.7);
 \draw[thick] (1.5,0)--(0.75,1.1)--(-0.5,1.35);
 \draw[thick] (-0.5,1.35)--(1.5,0)--(3.5,1.35)--(2.75,3.65)--(0.25,3.65)--(-0.5,1.35);
 \draw[thick] (0.75,1.1)--(2.25,1.1)--(2.75,2.45)--(1.5,3.4)--(0.25,2.45)--(0.75,1.1);
 \draw[thick] (1.5,0)--(2.25,1.1)--(3.5,1.35)--(2.75,2.45)--(2.75,3.65)--(1.5,3.4)--(0.25,3.65)--(0.25,2.45)--(-0.5,1.35);
  \end{tikzpicture}
   \end{minipage}
   \caption{The pentagonal pyramid (left, thick blue) with the truncation faces (in gray); on the right, its truncation, the pentagonal antiprism.}
  \end{figure}

 \end{ex}

 \begin{oss}
  As we noted the truncation of $\overline{\Gamma}$ is an ideal right-angled hyperbolic polyhedron; its volume can be calculated with a computer (see \cite{vesego} for a table with many computed volumes).
 \end{oss}

 To finish we prove that even though the rectification of $\Gamma$ is not a generalized hyperbolic polyhedron, nevertheless it is a limit of hyperideal (hence, proper) hyperbolic polyhedra.
 
 \begin{lem}\label{lem:convrett}
  There exists a continuous family of proper polyhedra $P_\epsilon$ with $1$-skeleton $\Gamma$ that converges to (a polyhedron in the isometry class of) $\overline{\Gamma}$ as $\epsilon\ra0$.
 \end{lem}
 \begin{proof}
  Take $P_\epsilon$ to be in the unique isometry class of polyhedra with $1$-skeleton $\Gamma$ and all angles equal to $\epsilon$; this can be done for any $\epsilon<\frac{\pi}{k}$ (where $k$ is the maximum valence among vertices of $\Gamma$) by Theorem \ref{teo:bonbao}.
  
  Since $P_\epsilon$ is defined only up to isometry, we need to fix an isometry class to establish the convergence. To do so we need a preliminary lemma of planar hyperbolic geometry.

 \begin{lem}\label{lem:conv2}
  If $\theta^n_1,\dots,\theta^n_k$ are angles of a hyperbolic $k$-gon $A_n$, for $n\in\mathbb{N}$, and as $n\ra\infty$ we have $\theta_i^n\ra \theta_i$ with $\sum_{i}\theta_i<(k-2)\pi$ then up to a suitable choice of isometry class of each $A_n$, every accumulation point of $A_n$ (with the convergence described in Section ...) is a hyperbolic polygon (possibly with ideal vertices); furthermore there is always at least one accumulation point.
 \end{lem}
 \begin{proof}
  Suppose first that $A_n$ is a triangle and $v_1^n,v_2^n,v_3^n$ are its vertices. Up to isometry we can suppose that for all $n$ the vertices $v_1^n$ and $v_2^n$ lie on the line $(-1,1)\times\{0\}\subseteq \mathbb{H}^2\subseteq\R^2$ and $v_3^n$ lies on the line $\{0\}\times\R$. Then as $n\ra \infty$, up to a subsequence if necessary, $v_i^n\ra v_i\in \overline{\mathbb{H}^2}$. 
  
  Pass to a subsequence such that $v_1^n,v_2^n,v_3^n$ converge to $v_1,v_2,v_3$.
  If the $v_i$'s are distinct, then $A_n$ converges to the triangle with vertices $v_1,v_2,v_3$. If $v_i=v_j$ we assume $i=1,j=2$ and $v_3\neq v_1$ (all other cases are identical). However this would mean that the internal angle of $A_n$ at $v_2^n$ must have the same limit as the external angle of $A_n$ at $v_1^n$. This implies that the limit of the sum of angles of $A_n$ is $\pi$, contradicting the hypothesis on the angles.
  
  Suppose now that $A_n$ is an $n$-gon, and triangulate it (with the same combinatorial triangulation for all $n$). One of the triangles must satisfy the hypotheses of the lemma, and thus we can choose an isometry class for each $n$ so that in any subsequence, this triangle converges to a non degenerate triangle; of course then if $A_{n_k}$ converges to some $A$, then $A$ must be non-degenerate since it contains a non-degenerate triangle.
 \end{proof}

 Now pick a vertex $v$ of $\Gamma$; up to isometry we can make it so that the corresponding vertex $v_\epsilon\in P_\epsilon$ is independent of $\epsilon$. Then any subsequence of $A_\epsilon:= P_\epsilon\cap \Pi_{v_\epsilon}$ satisfies the hypotheses of Lemma \ref{lem:conv2}, and we can choose a representative of $P_\epsilon$ so that (in this subsequence) $A_\epsilon$ converges to some $A$ of positive area.
 
 Each plane containing a face of $P_\epsilon$ is naturally an element of $(\rp)^*$, hence up to taking a subsequence each face of $P_\epsilon$ converges to a subset of a plane: the set of all limit planes will delimit a convex set $P^*$. To show that $P^*$ is a polyhedron, we need to prove that it is non-degenerate: this is because it must contain the cone from $v$ to $A$.
 
As we noted previously every vertex of $P^*$ must be a limit of vertices of $P_\epsilon$; likewise every edge of $P^*$ is a limit of edges of $P_\epsilon$. However a priori many vertices of $P_\epsilon$ could converge to the same vertex, and likewise with the edges. Furthermore, the angles of $P^*$ could be different than the limit of the angles of $P_\epsilon$. We now show however that in this specific case none of this happens.

Notice that each edge $e_\epsilon$ of $P_\epsilon$ has its dihedral angle converging to $0$, hence the faces $\Pi^1_\epsilon$ and $\Pi^2_\epsilon$ determining $e$ must either coincide in the limit (which would contradict the fact that $P^*$ is non-degenerate) or must converge to two faces $\Pi_1$ and $\Pi_2$ intersecting in an edge tangent to $\partial \h$. Therefore, every edge of $P^*$ is tangent to $\partial\h$ and $P^*$ is the rectification of its $1$-skeleton. If there are no two vertices of $P_\epsilon$ converging to the same vertex of $P^*$ then of course the $1$-skeleton of $P^*$ is $\Gamma$ and we have concluded. Suppose then that many vertices of $P_\epsilon$ converge to a vertex $w$ of $P^*$: since each vertex of $P_\epsilon$ is outside $\h$ and each edge intersects $\h$, $w$ must lie on the sphere at infinity. Then $P^*$ is a non-degenerate projective polyhedron such that:
\begin{itemize}
 \item each of its edges is tangent to $\partial\h$;
 \item some of its vertices are in $\partial\h$.
\end{itemize}

However such a polyhedron does not exist. Therefore $P^*=\overline{\Gamma}$.

Suppose now that $P_\epsilon$ has a different subsequence converging to some other $\tilde{P^*}$. The same argument implies that this subsequence converges to $\overline{\Gamma}$, hence $P_\epsilon\ra\overline{\Gamma}$.
 \end{proof}

 \subsection{The Schl\"afli formula and volume}

The fundamental tool to study the volume of proper polyhedra is the Schl\"afli formula (see for example \cite[Chapter: The Schl\"afli differential equality]{miln}). Remember that the volume of a proper polyhedron is equal to the hyperbolic volume of its truncation.

\begin{teo}\label{teo:schlafli}
 If $P_t$ is a smooth family of proper or almost proper polyhedra without ideal vertices with the same $1$-skeleton, same set of almost proper vertices, dihedral angles $\theta_1^t,\dots,\theta_n^t$ and edge lengths $l_1^t,\dots,l_n^t$, then
 \begin{displaymath}
  \frac{\partial\vol(P_t)}{\partial t}_{\big\rvert_{t=t_0}}=-\frac{1}{2}\sum_i l_i^{t_0}\frac{\partial\theta^t_i}{\partial t}_{\big\rvert_{t=t_0}}
 \end{displaymath}
\end{teo}
\begin{oss}
 Usually this result is stated for compact hyperbolic polyhedra (i.e. with no hyperideal vertices), however it is straightforward to generalize it to proper or almost proper polyhedra without ideal vertices (it can also be generalized to the case with ideal vertices; however this case carries some technical difficulties and we will not need it). This is because a path of proper or almost proper polyhedra $P_t$ induces a path of truncated polyhedra $P_t^0$; the length of an edge in $P_t$ is the same (by definition) as the length of the corresponding edge in the truncation, and the derivatives of all angles are the same, since an angle of $P_t^0$ is either the same as the corresponding angle of $P_t$, differs by a constant, or is constantly $\frac{\pi}{2}$ and thus does not contribute to the sum.
 
 However the Schl\"afli formula definitely does not hold for generalized hyperbolic polyhedra, as the truncation angles could vary. 
\end{oss}

In particular, the Schl\"afli formula implies that if all the angles are \emph{decreasing}, the volume is \emph{increasing}. 

We recall a well known fact about the convergence of volumes.

\begin{lem}\label{lem:continuity}
 If $P_n\subseteq \h, n\in\N$ is a sequence of compact hyperbolic polyhedra, and $P_n$ converges to the compact convex set $P$ as $n\ra+\infty$ (as elements of ${({(\rp)}^*)}^F$), then $\vol(P_n)\ra\vol(P)$.
\end{lem}
\begin{proof}
 Suppose first that $P$ has non-empty interior (i.e. it is a polyhedron).
 Up to isometry, we can assume that $0$ is in the interior of $P$ (therefore, $0\in P_n$ for $n$ big enough). Consider the map $\Phi^0_\lambda:\R^3\ra \R^3$ given by multiplication by $\lambda$. View $P_n$ and $P$ as subsets of $\h\subseteq \R^3$. Since $P_n\ra P$, there exists an $\epsilon_n$ such that $\Phi_{1-\epsilon_n}(P_n)\subseteq P$ with $\epsilon_n\ra0$ as $n\ra\infty$, because $P$ has non-empty interior. By modifying $\epsilon_i$ for $i\leq n$  appropriately, we can assume that $\Phi_{1-\epsilon_n}(P_n)$ contains every $\Phi_{1-\epsilon_i}(P_i)$ for $i<n$ (since $P_n$ is convex for every $n$). Therefore the sequence $\Phi_{1-\epsilon_n}(P_n)$ is an increasing sequence of compact (convex) polyhedra that converges to $P$; the monotone convergence theorem then tells us that $\vol\left(\Phi_{1-\epsilon_i}(P_i)\right)\ra \vol(P)$. Since $\epsilon_i\ra0$ then $\vol\left(\Phi_{1-\epsilon_i}(P_i)\right)-\vol(P_i)$ converges to $0$, which concludes the proof.
 
 If instead $P$ has volume $0$, then for any $\epsilon>0$, there is an $\overline{n}$ big enough that $P_n$ is contained in an $\epsilon$ neighborhood of $P$ for any $n>\overline{n}$, which implies that $\vol(P_n)\ra \vol(P)=0$.
\end{proof}

Notice that in general it is \emph{not true} that if $P_n$ is a sequence of compact hyperbolic polyhedra that converges to a (non-degenerate) \emph{finite-volume} $P$, then $\vol(P_n)\ra\vol(P)$. To see a counterexample, take a tetrahedron $T_\alpha$ with $3$ real vertices and a hyperideal vertex, with sum of angles at the hyperideal vertex equal to $\alpha<\pi$. Glue two copies of $T_\alpha$ along the truncation face to obtain a compact prism $P_\alpha$, and let $\alpha\ra\pi$. Then $\vol(P_\alpha)=2\vol(T_\alpha)\ra2\vol(T_\pi)$ but $P_\alpha$ converges to a tetrahedron (where one copy of $T_\alpha$ gets pushed to infinity), and thus $\vol(P_\pi)=\vol(T_\pi)$.

However, if $P_n$ converges to a finite-volume polyhedron $P$ and the convergence is nice enough, then the result still holds; this is the content of the following lemma.

\begin{lem}\label{lem:continuity2}
 Suppose $P_n$ is a sequence of compact polyhedra (i.e. all their vertices are real) with $1$-skeleton $\Gamma$ such that $P_n\ra P$ with $P$ a finite volume (possibly degenerate) convex subset of $\h$ and such that if $v^n_1,v^n_2$ are vertices of $P_n$ that converge to the same point in $P$, then the distance between $v_1^n$ and $v_2^n$ converges to $0$. Then $\vol(P_n)\ra\vol(P)$.
\end{lem}
\begin{proof}
If $P$ is a point on $\partial\h$ then every edge length of $P_n$ goes to $0$ by assumption; this implies that $P_n$ is contained in an $\epsilon$-neighborhood of one of its vertices, hence its volume goes to $0$. Suppose then that $P$ contains a point $p\in\h$.

Take any subsequence $P_{n_k}$ of $P_n$; up to a subsequence $P_{n_{k_j}}$ we can assume that every vertex of $P_{n_{k_j}}$ converges. If we prove that $\vol\left(P_{n_{k_j}}\right)\ra\vol(P)$ we conclude the proof; therefore we can assume that every vertex of $P_n$ converges.

 It suffices to prove the lemma for tetrahedra. Indeed we can triangulate $P_n$ by taking the cone over an interior point $p_n$ converging to $p$, and by further triangulating each pyramid in a fixed combinatorial way. Each tetrahedron in the triangulation converges to a (maybe degenerate) tetrahedron in a triangulation of $P$, hence if the volume of each tetrahedron converges to the volume of its limit, then the volume of $P_n$ must converge to the volume of $P$.
 
 So let $P_n$ be a tetrahedron converging to a (possibly degenerate) tetrahedron $T$. If $T$ is compact we can apply Lemma \ref{lem:continuity} to conclude. If $T$ is non-degenerate and has ideal points, we can for the sake of simplicity further triangulate so that $T$ only has one ideal vertex $v$. Then if we truncate $P_n$ and $T$ along a horosphere centered in $v$ we obtain a sequence of compact sets with volumes close to those of $P_n$ converging to a compact set with volume close to that of $T$, which concludes the proof.
 
 If instead $T$ is degenerate and has ideal vertices, we distinguish $2$ cases. If no two vertices of $P_n$ converge to the same point, then we can apply the same reasoning as the previous case by cutting along appropriate horospheres. Otherwise, we need to show that $\vol(P_n)\ra 0$. There are (at least) two vertices of $P_n$ converging to the same vertex of $T$, and by assumption their distance goes to $0$. This means that for $n$ big enough, $P_n$ is contained in an $\epsilon$-neighborhood of one of its faces. Since $P_n$ has bounded surface area, $\vol(P_n)\ra 0$.
\end{proof}

\section{The maximum volume theorem}\label{sec:maxvol}

\subsection{Related results}

Before we delve into the proof of Theorem \ref{maxvol}, a few words are necessary on what makes this result particularly complicated.

\textbf{$\mathcal{A}_\Gamma$ contains polyhedra with obtuse angles.}

Often when hyperbolic polyhedra are considered, they are restricted to have acute angles (especially when studying their relationship to orbifolds and cone-manifolds). This limitation greatly simplifies matters, mainly because of two properties of acute-angled polyhedra, both consequences of Andreev's Theorem \cite{andreev}:

\begin{itemize}
 \item Dihedral angles are global coordinates for acute-angle polyhedra;
 \item The space of dihedral angles of acute-angle polyhedra is a convex subset of $\R^N$ (except when the polyhedron has $4$ vertices).
\end{itemize}

Using these properties, it is a just a matter of applying the Schl\"afli identity \ref{teo:schlafli} to prove that the maximum volume of acute-angled polyhedra is the volume of the rectification.

By contrast, if we allow obtuse angles things are considerably more difficult. It is unknown whether dihedral angles determine a polyhedron; furthermore the space of dihedral angles is never convex \cite{diaz}.

On the other hand, requiring that a polyhedron be acute-angled is very restricting; in particular, there are no non-simple acute-angle compact polyhedra. Therefore allowing obtuse angles greatly increases the scope of Theorem \ref{maxvol}.

\textbf{$\mathcal{A}_\Gamma$ contains polyhedra with any combination of real, ideal or hyperideal vertices.}

The case of polyhedra with only ideal vertices has been known for a long time.

\begin{teo}\cite[Theorem 14.3]{rivinvol}\label{teo:rivvol}
 There is a unique (up to isometry) hyperbolic ideal polyhedron with fixed $1$-skeleton of maximal volume; furthermore this polyhedron is maximally symmetric.
\end{teo}

Notice that even though the maximal volume polyhedron is unique, the symmetry property does not always determine it uniquely. 

Theorem \ref{teo:rivvol} relies on the fact that ideal polyhedra share the two above properties of acute-angled polyhedra: they are determined by their angles, and the space of dihedral angles is a convex polytope. Moreover, this result relies on the concavity of the volume function for ideal polyhedra; this result once again does not hold for compact polyhedra.

We might wish to extend Theorem \ref{teo:rivvol} to polyhedra with both real and ideal vertices; a pleasant result would be the following:

\textbf{``Theorem''}: For any polyhedron $P$ with real and ideal vertices, there is a polyhedron $Q$ with only ideal vertices, with the same $1$-skeleton as $P$ and such that $\vol(P)\leq\vol(Q)$.

This would imply that the maximal volume ideal polyhedron is also of maximal volume among polyhedra with both real and ideal vertices. Unfortunately, this ``Theorem'' has no chance of being true: there are some graphs which are the $1$-skeleton of hyperbolic polyhedra but are not the $1$-skeleton of any ideal polyhedron (a complete characterization of inscribable graphs, i.e. those that are the $1$-skeleton of an ideal polyhedron, can be found in \cite{rivin}). In such a case, the maximal volume polyhedron (if it even existed) would have some ideal vertices and some real vertices, would possibly not be unique and would certainly be very difficult to determine.

Therefore, admitting vertices of all $3$ types allows us to obtain a satisfying result, where the maximal volume is obtained at a very concrete polyhedron whose volume can be explicitly computed. However it comes at the cost of increased complications in the proofs, mainly having to deal with almost proper polyhedra.

\textbf{It is not a simple application of the Schl\"afli identity.}

Looking at the Schl\"afli identity might suggest that it immediately implies the Maximum Volume Theorem. An argument might go like this:

\emph{A local maximum must have every length of an edge (not arising from truncation) equal to $0$; therefore it must be the rectification.}

This does not work for two reasons:

\begin{itemize}
\item Simply putting derivatives equal to $0$ would give local maxima in the interior of $\mathcal{A}_\Gamma$; to truly find a supremum we would have to analyze the behavior at its boundary.
 \item Dihedral angles do not give local coordinates on all of $\overline{\mathcal{A}_\Gamma}$; in particular it is unknown whether they would give local coordinates at the boundary $\partial\mathcal{A}_\Gamma$ (if, for example, some angles are equal to $0$).
\end{itemize}

\subsection{Proof of the maximum volume theorem}

The following is the main result of the paper.
\begin{teo}\label{maxvol}
For any planar $3$-connected graph $\Gamma$,
 $$\vol\left(\overline{\Gamma}\right)=\sup_{P\in\mathcal{A}_\Gamma}\mathrm{Vol}(P).$$
\end{teo}
\begin{proof}
 For the sake of understanding, the proof of the key Proposition \ref{prop:degen} is postponed in Subsection \ref{sec:proofs}.
 
 It is easy to see that $\vol\left(\overline{\Gamma}\right)\leq \sup_{P\in\mathcal{A}_\Gamma}\mathrm{Vol}(P)$; the family $P_\epsilon$ defined in the statement of Lemma \ref{lem:convrett} is a family in $\mathcal{A}_\Gamma$, has decreasing angles (hence increasing volume) and it converges to $\overline{\Gamma}$. Since the length of every edge of $P_\epsilon$ converges to $0$ (because it is equal to the distance between two planes who converge to be asymptotic), we can apply Lemma \ref{lem:continuity2} to the truncation of $P_\epsilon$ to obtain that $\vol(P_\epsilon)\ra\vol(\overline{\Gamma})$.

 Now we prove that $\vol(P)\leq \vol\left(\overline{\Gamma}\right)$ for any $P\in\mathcal{A}_\Gamma$. We first assume that $P$ only has hyperideal vertices.

 Let $(\theta_1,\dots,\theta_k)$ be the dihedral angles of $P$. Since the space of dihedral angles of hyperideal polyhedra is convex by Theorem \ref{teo:bonbao}, there is a continuous family of polyhedra with decreasing angles connecting $P$ to $P_\epsilon$ (for a small enough $\epsilon$), which implies that $\vol(P)<\vol(P_\epsilon)$. But we have already seen that $\vol(P_\epsilon)$ increases to $\vol\left(\overline{\Gamma}\right)$ as $\epsilon\ra 0$, which implies that $\vol(P)<\vol\left(\overline{\Gamma}\right)$.
 
 Let now $P$ be any polyhedron in $\mathcal{A}_\Gamma$. We wish to define inductively a (possibly empty) sequence of polyhedra $P^{(1)},\dots,P^{(m)}$ without ideal vertices such that 
 \begin{equation}\label{eq:chain}
  \vol(P)<\vol\left(P^{(1)}\right)<\dots<\vol\left(P^{(m)}\right)<\vol\left(\overline{\Gamma}\right).
 \end{equation}

 The key proposition in building this chain is the following.

\begin{prop}\label{prop:degen}
 Let $P$ be either a proper or an almost proper polyhedron with $1$-skeleton $\Gamma$ with no ideal vertices and at least one real vertex. Then there exists a generalized hyperbolic polyhedron $P^*$ with the following properties:
 \begin{itemize}
  \item $\vol(P^*)>\vol(P)$;
  \item $P^*$ is either proper or almost proper;
  \item $P^*$ has at most the same number of real vertices as $P$;
  \item $P^*$ either has fewer vertices, fewer real vertices or fewer proper vertices than $P$;
  \item if $P^*$ has ideal vertices, then it has exactly one;
  \item if $P^*$ has ideal vertices, then it is proper.
\item the $1$-skeleton of $P^*$ can be obtained from $\Gamma$ via a finite sequence of the following moves:
   \begin{enumerate}[label=(\roman*)]
    \item an edge of $\Gamma$ collapses to a vertex (see Figure \ref{fig:edgecollapse});
    \item a face of $\Gamma$ collapses to an edge (see Figure \ref{fig:facecollapse}).
   \end{enumerate}
    \end{itemize}
\end{prop}

The proof of this proposition is postponed to Section \ref{sec:proofs}.

  Let now $P$ be a proper polyhedron with $1$-skeleton $\Gamma$.
  
  If $P$ only has hyperideal vertices, then we have already seen that $\vol(P)<\vol\left(\overline{\Gamma}\right)$.
  
  If $P$ has some ideal vertices, then take the polyhedron $P':=\Phi_\lambda^v(P)$ for $\lambda$ slightly larger than $1$ and $v\in P$ where $\Phi_\lambda^v$ is the homothety of center $v$ and factor $\lambda$. For any $\epsilon>0$ there is a $\lambda>1$ and close enough to $1$ such that $P'$ is proper, it has no ideal vertices and its volume is larger than $\vol(P)-\epsilon$. To see this, notice that as $\lambda\ra 1$, the truncation face of a hyperideal vertex that becomes ideal is a hyperbolic polygon that becomes Euclidean (hence, the length of its sides goes to $0$) and thus we can apply Lemma \ref{lem:continuity2}) to obtain that the volume changes continuously. If we prove that $\vol(\Phi_\lambda^v(P))<\vol\left(\overline{\Gamma}\right)$ for any $\lambda$, then we also prove that $\vol(P)\leq\vol\left(\overline{\Gamma}\right)$ which implies the theorem. Therefore we can assume that $P$ has no ideal vertices.
  
If $P$ has no ideal vertices and at least one real vertices, take $P^*$ given by Proposition \ref{prop:degen}. If $P^*$ does not have ideal vertices, take $P^{(1)}:=P^*$, if it does then take $P^{(1)}:=\Phi_\lambda^v(P^*)$ with $v\in P^*$ and $\lambda>1$ small enough to ensure that $\vol\left(P^{(1)}\right)>\vol(P)$.
  
  Suppose now that we have defined the sequence $P^{(1)},\dots,P^{(j)}$ satisfying the inequalities of \eqref{eq:chain}, and define inductively $P^{(j+1)}$.
  
  By hypothesis $P^{(j)}$ cannot have ideal vertices. If $P^{(j)}$ has no real vertices, we stop. If $P^{(j)}$ has some real vertices, then we can apply Proposition \ref{prop:degen} to it and we define $P^{(j+1)}$ to be the resulting polyhedron (once again applying a small expansion if it has ideal vertices).
 
 Notice that when passing from $P^{(i)}$ to $P^{(i+1)}$ the tuple $$\left(\#\textrm{ of vertices},\#\textrm{ of real vertices},\#\textrm{ of proper vertices}\right)$$ decreases in the lexicographic order, since we applied Proposition \ref{prop:degen} to go from $P^{(i)}$ to $P^{(i+1)}$. Therefore at some point we have to arrive at a $P^{(m)}$ with only hyperideal vertices: we have seen at the beginning of the proof that this implies $\vol(P^{(m)})<\vol\left(\overline{\Gamma'}\right)$ where $\Gamma'$ is the $1$-skeleton of $P^{(m)}$.
 
 The conclusion of the proof comes from the following proposition, which could be of independent interest.

 \begin{prop}\label{prop:edgecollapse}
  If $\Gamma'$ is obtained from $\Gamma$ either via a single edge collapsing to a vertex (see Figure \ref{fig:edgecollapse}) or a single face collapsing to an edge (see Figure \ref{fig:facecollapse}), then $\vol\left(\overline{\Gamma'}\right)\leq\vol\left(\overline{\Gamma}\right)$.
 \end{prop}
\begin{proof}
The proof works exactly the same for either the edge collapse or the face collapse; we carry it out for the edge collapse.

The key observation (see Remark \ref{oss:dual}) is that $\vol\left(\overline{\Gamma}\right)=\vol\left(\overline{\Gamma^*}\right)$, and if $\Gamma'$ is obtained from $\Gamma$ by an edge collapse, then $\Gamma'^*$ is obtained from $\Gamma^*$ by deleting an edge $e$ (see Figure \ref{fig:dualcollapse}).

\begin{figure}
   \centering
  \begin{minipage}{.45\textwidth}
   \centering
    \begin{tikzpicture}
 \draw[thick] (0,2)--(1,1)--(0,0);
 \draw[thick](0,1)--(4,1);
 \draw[thick] (4,2)--(3,1)--(4,0);
 \draw[thick, red] (0.3,1.3)--(2,2.5)--(3.7,1.3)--(3.7,0.7)--(2,-0.5)--(2,2.5);
 \draw[thick,red] (0.3,1.3)--(0.3,0.7)--(2,-0.5);
 \end{tikzpicture}
  \end{minipage}
$\longrightarrow$
  \begin{minipage}{.45\textwidth}
   \centering
    \begin{tikzpicture}
 \draw[thick] (0,2)--(1,1)--(0,0);
 \draw[thick](0,1)--(2,1);
 \draw[thick] (2,2)--(1,1)--(2,0);
 \draw[thick, red] (0.3,1.3)--(1,2)--(1.7,1.3)--(1.7,0.7)--(1,0);
 \draw[thick,red] (0.3,1.3)--(0.3,0.7)--(1,0);
 \end{tikzpicture}
  \end{minipage}
  \caption{If an edge collapses, dually an edge gets deleted.}\label{fig:dualcollapse}
\end{figure}
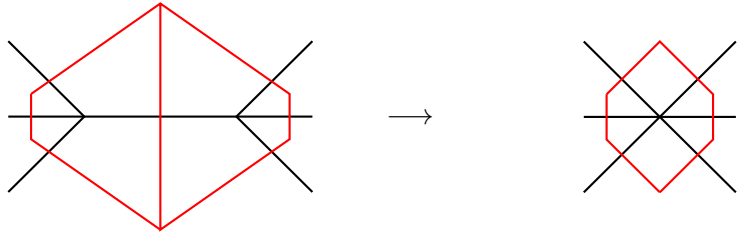

Let now $P_{\epsilon,\alpha},$ be the polyhedron with $1$-skeleton $\Gamma^*$, dihedral angle $\alpha$ at the edge $e$ and every other dihedral angle equal to $\epsilon$. Because of Theorem \ref{teo:bonbao}, for any $\epsilon$ sufficiently close to $0$, $P_{\epsilon,\alpha}$ exists for $\alpha\in(0,\pi-k\epsilon)$ where $k$ depends only on the valence of $\Gamma^*$ at the endpoints of $e$. Because of the Schl\"afli formula, $\vol(P_{\epsilon,\alpha})$ is decreasing in $\alpha$, and in particular $\vol(P_{\epsilon,\pi-k\epsilon})<\vol(P_{\epsilon,\epsilon})$. But  $\vol(P_{\epsilon,\epsilon})\ra \vol\left(\overline{\Gamma}\right)$ and $\vol(P_{\epsilon,\pi-k\epsilon})\ra \vol\left(\overline{\Gamma'}\right)$. The first convergence comes from Lemmas \ref{lem:convrett} and \ref{lem:continuity2}; the second convergence comes with the same argument as in the proof of Lemma \ref{lem:convrett} after we notice that if two faces have an angle converging to $\pi$, then they must converge to the same plane.
\end{proof}

\begin{oss}
 Proposition \ref{prop:edgecollapse} could be translated into a statement about ideal right-angled polyhedra. Specifically, it says that if $P_1$ and $P_2$ are two ideal right-angled polyhedra with $1$-skeleta $\Gamma_1$ and $\Gamma_2$ related as in Figure \ref{fig:rettangolo}, then $\vol(P_1)\geq\vol(P_2)$.
\end{oss}

 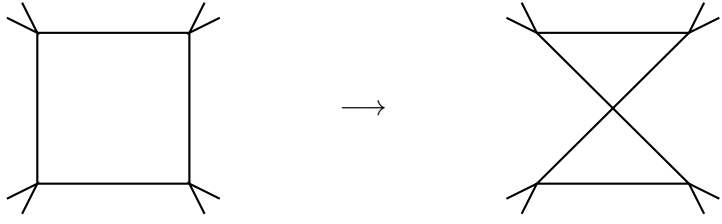
\begin{figure}
  \centering
  \begin{minipage}{.45\textwidth}
   \centering
 \begin{tikzpicture}
 \draw[thick] (3,2)--(3,0)--(1,0)--(1,2)--(3,2);
 \draw[thick] (.6,2.2)--(1,2)--(.8,2.4);
 \draw[thick] (3.4,2.2)--(3,2)--(3.2,2.4);
 \draw[thick] (.6,-.2)--(1,0)--(.8,-.4);
 \draw[thick] (3.4,-.2)--(3,0)--(3.2,-.4);
 \end{tikzpicture}
  \end{minipage}
$\longrightarrow$
  \begin{minipage}{.45\textwidth}
   \centering
 \begin{tikzpicture}
 \draw[thick] (3,2)--(1,0)--(3,0)--(1,2)--(3,2);
 \draw[thick] (.6,2.2)--(1,2)--(.8,2.4);
 \draw[thick] (3.4,2.2)--(3,2)--(3.2,2.4);
 \draw[thick] (.6,-.2)--(1,0)--(.8,-.4);
 \draw[thick] (3.4,-.2)--(3,0)--(3.2,-.4);
 \end{tikzpicture}
  \end{minipage}
  \caption{Two edges on the same face getting ``switched''.}\label{fig:rettangolo}
 \end{figure}

\end{proof}

\begin{cor}\label{voldual}
 \begin{displaymath}
  \sup_{P\in\mathcal{A}_\Gamma}\vol(P)=\sup_{P'\in\mathcal{A}_{\Gamma^*}}\vol(P')
 \end{displaymath}

\end{cor}
\begin{proof}
 As we noted before, the rectification of $\Gamma$ and the rectification of $\Gamma^*$ have isometric truncations.
\end{proof}

\subsection{Proof of Proposition \ref{prop:degen}}\label{sec:proofs}

We are going to prove Proposition \ref{prop:degen} by deforming the polyhedron $P$ until it collapses to a limit polyhedron. First we need two lemmas describing what happens when the limit polyhedron has ideal vertices. Both follow the same general reasoning (and similar notation) of Proposition 5 and Lemma 22 of $\cite{bonbao}$.

\begin{lem}\label{lem:anglesum}
 Let $P_n$ be a sequence of proper or almost proper polyhedra with $1$-skeleton $\Gamma$ and with angles bounded away from $0$ and $\pi$, and suppose $P_n$ converges as $n\ra +\infty$ to the projective polyhedron $P^*$. Further suppose that all the vertices of $P_n$ also converge, and a sequence of vertices $v_n\in P_n$ converges to $v\in\partial\h$. Let $e_1,\dots,e_k$ be the edges of $\Gamma$ with exactly one endpoint converging to $v$, and let $e_1,\dots,e_{k'}$ be the subset of those edges converging to a segment intersecting $\h$. Then the sum of external dihedral angles of $e_1,\dots,e_{k'}$ converges to either $2\pi$ (if and only if $k=k'$) or $\pi$ otherwise.
\end{lem}

 \begin{proof}
 First notice that all edges of a limit of proper or almost proper polyhedra must either intersect $\h$ or be tangent to its boundary.
 
  Let $e_1^n,\dots,e_k^n$ be the edges of $P^n$ corresponding to $e_1,\dots,e_k$ respectively, each converging to $e^\infty_1,\dots,e^\infty_k$ (many different $e^\infty_i$s could be in the same edge of $P^*$). We distinguish several cases.
  
  \textbf{Case 1.} The segment $e^\infty_i$ intersects $\h$ for every $i$.
  
  In this case clearly $k=k'$.
  Let $F^n_1,\dots,F_k^n$ be the faces of $P_n$ containing $e_1^n,\dots,e_k^n$. Let $\Pi_j^n$ be the plane containing $F_j^n$ and denote with $\Pi_j^\infty$ their limit. Finally let $\tilde{\theta}^n_1,\dots,\tilde{\theta}^n_k$ be the external dihedral angles of $e_1^n,\dots,e_k^n$ respectively. 
  
  Since dihedral angles between two planes vary continuously (as long as their intersection is contained in $\h$), $\lim_{n\ra \infty}\tilde{\theta}_j^n=\tilde{\theta}_j^\infty$ where $\tilde{\theta}_j^\infty$ is the angle between planes $\Pi^\infty_j$ and $\Pi^\infty_{j+1}$. 
  Then the second part of \cite[Proposition 5]{bonbao} shows that 
  $$\sum_i\tilde{\theta}^\infty_{i}=2\pi$$
  
  which concludes the proof in this case. 
  
  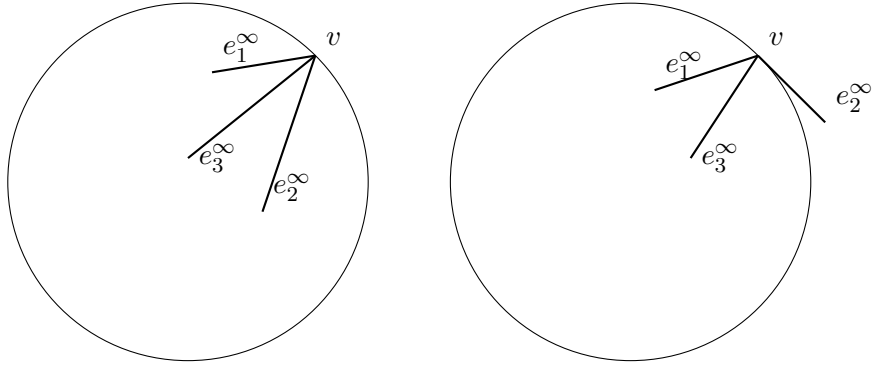
\begin{figure}
   \centering
\begin{minipage}{.45\textwidth}
 
 \begin{tikzpicture}[scale=0.79]
  \draw [thin] (1,1) circle[radius=3cm];
  \draw [thick] (3.12,3.12) node[above right]{$v$}--(1.40,2.84) node[above right]{$e_1^\infty$};
  \draw [thick] (3.12,3.12) --(2.24,0.5) node[above right]{$e_2^\infty$};
  \draw [thick] (3.12,3.12)--(1,1.4) node[right]{$e_3^\infty$};
 \end{tikzpicture}

\end{minipage}
   \begin{minipage}{.45\textwidth}
 \begin{tikzpicture}[scale=0.79]
  \draw [thin] (1,1) circle[radius=3cm];
  \draw [thick] (3.12,3.12) node[above right]{$v$}--(1.40,2.54) node[above right]{$e_1^\infty$};
  \draw [thick] (3.12,3.12) --(4.24,2.0) node[above right]{$e_2^\infty$};
  \draw [thick] (3.12,3.12)--(2,1.4) node[right]{$e_3^\infty$};
 \end{tikzpicture}
   \end{minipage}
\caption{Cases 1 and 2}

  \end{figure}

  \textbf{Case 2.} Some edges $e^\infty_i$ intersect $\h$, while at least some other $e^\infty_j$ is tangent to $\partial\h$; however, every limit edge tangent to $\partial\h$ is in the same edge of $P^*$.
  
  Let $v_j^n$ be the endpoint of $e_j^n$ converging to $v_j^\infty\neq v$; $v_j^\infty$ is necessarily hyperideal; we may suppose also that each $v_j^n$ is hyperideal. Truncate $P_n$ along $\Pi_{v_j^n}$ and double it along the truncation face. This gives a sequence $\tilde{P}_n$ converging to the polyhedron obtained from $P^*$ by truncating along $\Pi_{v_j^\infty}$ and doubling. This sequence falls under Case 1, and the sum of the external dihedral angles of $\tilde{P_n}$ at the edges converging to $v$ must be equal to $2\sum_{j=1}^{k'}\tilde{\theta}_j$, and it must converge to $2\pi$ which gives the thesis.

  \textbf{Case 3.} Some edges $e^\infty_i$ intersect $\h$, while at least some other $e^\infty_j$ is tangent to $\partial\h$; morever, there are at least two distinct edges of $P^*$ tangent to $\partial\h$.
  
  This case (like the following one) is essentially the same as the corresponding case of Lemma 22 of \cite{bonbao}.
  This means that $P^*$ has two (or more) edges tangent to $\partial\h$ in $v$ and some other edge with endpoint $v$ which intersects $\h$. Then let $S$ be a horosphere centered in $v$ and consider $A:=S\cap P^*$. The polygon $A$ is not compact, has at least two ends (one in the direction of each edge tangent to $\partial\h$) and at least one vertex with positive angle, but this is impossible, hence this case never happens.
  
  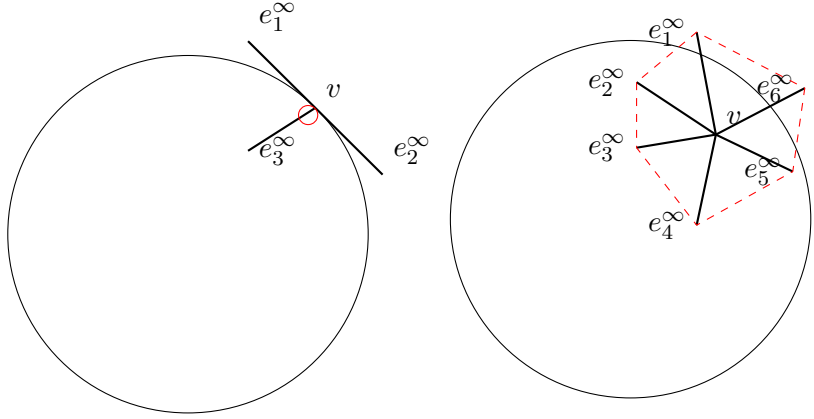
\begin{figure}
   \centering
\begin{minipage}{.45\textwidth}
 
 \begin{tikzpicture}[scale=0.79]
  \draw [thin] (1,1) circle[radius=3cm];
  \draw [thick] (3.12,3.12) node[above right]{$v$}--(2.0,4.24) node[above right]{$e_1^\infty$};
  \draw [thick] (3.12,3.12) --(4.24,2.0) node[above right]{$e_2^\infty$};
  \draw [thick] (3.12,3.12)--(2,2.4) node[right]{$e_3^\infty$};
  \draw[red] (3,3) circle[radius=.16cm];
 \end{tikzpicture}

\end{minipage}
   \begin{minipage}{.45\textwidth}
   
 \begin{tikzpicture}[scale=0.79]
  \draw [thin] (1,1) circle[radius=3cm];
  \draw [thick] (2.42,2.42) node[above right]{$v$}--(2.1,4.14) node[left]{$e_1^\infty$};
  \draw [thick] (2.42,2.42) -- (1.1,2.2) node[left]{$e_3^\infty$};
  \draw [thick] (2.42,2.42) -- (2.1,0.9) node[left]{$e_4^\infty$};
  \draw [thick] (2.42,2.42) -- (3.7,1.8) node[left]{$e_5^\infty$};
  \draw [thick] (2.42,2.42) -- (3.9,3.2) node[left]{$e_6^\infty$};
  \draw [thick] (2.42,2.42) -- (1.1,3.3) node[left]{$e_2^\infty$};
  \draw [dashed, red] (1.1,3.3)--(1.1,2.2)--(2.1,0.9)--(3.7,1.8)--(3.9,3.2)--(2.1,4.14)--(1.1,3.3);
 \end{tikzpicture}

   \end{minipage}
\caption{Cases 3 and 4}

  \end{figure}
  \textbf{Case 4.} Every edge $e^\infty_i$ is tangent to $\partial\h$.
  
  This case is impossible as well: each $e^n_i$ has one endpoint $v^n_i$ that does not converge to $v$. Since $\Gamma$ is $3$-connected, there are three indices $a,b,c$ such that $e_a,e_b,e_c$ have the endpoints $v_a,v_b,v_c$ (the one not converging to $v$) which are all different. Then the lines connecting $v_a^n,v_b^n$ and $v_c^n$ must intersect $\h$, but this is impossible since they converge to points each lying on a different line tangent to $\partial\h$ in the same point $v$. This contradicts the fact that $P_n$ is a generalized hyperbolic polyhedron.

 \end{proof}
 
 \begin{cor}
  With the same hypotheses of Lemma \ref{lem:anglesum}, if a vertex of $P_n$ converges to a point on $\partial\h$, then it converges to a vertex of the limit polyhedron.
 \end{cor}
\begin{proof}
 Suppose $v_n$ is a sequence of vertices converging to $v\in P^*$ and $v$ is not a vertex of $P^*$; then it is either contained in the interior of an edge or the interior of a face. If it is contained in the interior of an edge, the analysis of Case 3 of Lemma \ref{lem:anglesum} show that there is a contradiction; likewise if it is contained in the interior of a face there is a contradiction because of the reasoning in Case 4 of Lemma \ref{lem:anglesum}.
\end{proof}

 Further notice that from the hypotheses of Lemma \ref{lem:anglesum} we can drop the assumption that the vertices of $P_n$ converge, since we can always pass to a subsequence such that this holds (and using the easy fact that if every subsequence of a sequence has a subsequence converging to $x$, the whole sequence converges to $x$).

\begin{lem}\label{lem:hypvert}
 Take a sequence of proper or almost proper polyhedra $P_n$ with $1$-skeleton $\Gamma$ and decreasing angles bounded away from $0$, such that $P_n$ converges to the projective polyhedron $P^*$ as $n\ra +\infty$. If there is a vertex $v_n$ of $P_n$ converging to an ideal vertex $v$ of $P^*$, then $v_n$ is a real vertex of $P_n$ for all $n$.
\end{lem}
\begin{proof}

We proceed by contradiction, supposing that $v_n$ is hyperideal (notice that if $v_{\overline{n}}$ is ideal or hyperideal for some $\overline{n}$ then it is strictly hyperideal for all $n>\overline{n}$ by Lemma \ref{lem:angid}).

 Let $K_v\subseteq \Gamma\subseteq S^2$ be the union of all vertices and edges converging to $v$, and consider as we did before the edges $e_1,\dots,e_k$ of $\Gamma$ with exactly one endpoint in $K_v$. Denote with $e_1^n,\dots,e_k^n$ the corresponding edges of $P_n$ and call $e_1^\infty,\dots,e^\infty_k$ the limit edges. As before the external dihedral angle of $e_i^n$ is denoted $\tilde{\theta}_i^n$. Let $F_1^n,\dots,F_k^n$ be the faces containing these edges. Each $F_i^n$ lies on a plane $\Pi_i^n\subseteq \R^3$ and determines a half space $H_i^n\subseteq \R^3$ bounded by $\Pi_i^n$ (the one that contains $P_n$). Pick $n$ big enough so that the triple intersections of the various $\Pi_i^n$s are very close to $v$.
 
 We distinguish the same $4$ cases as in Lemma \ref{lem:anglesum}; we showed that Case 3 and Case 4 are impossible (under more lax assumptions) hence we skip them.
 
 \textbf{Case 1.} Every limit edge $e_i^\infty$ intersects $\h$.
 
 We show that $\sum_i\tilde{\theta}_i^n>2\pi$. Since the internal dihedral angles are decreasing, the external dihedral angles are increasing in $n$, which would contradict Lemma \ref{lem:anglesum}.
 
 Let $Q_n$ be the intersection of all $H_i^n$s; this is a convex non-compact subset of $\R^3$. By assumption $Q_n$ has some hyperideal vertices (since certainly $P_n\subseteq Q_n$ and $P_n$ has a hyperideal vertex close to $v$).
 
 Let $\Upsilon\subseteq \R^2$ be the $1$-skeleton of $Q_n$. It must have (unbounded) edges $e_1,\dots,e_k$ and additional edges $e_1',\dots,e_l'$, however it cannot have any cycle since an innermost cycle would bound some face of $Q_n$ not contained in $\Pi_1,\dots,\Pi_k$. 
 
 If $\gamma\subseteq \R^2$ is an embedded $1$-manifold intersecting transversely distinct edges $e_{i_1},\dots,e_{i_j}$ exactly once, we write $\sum_{\gamma}\tilde{\theta}_e$ for the sum $\sum_{m=1}^j\tilde{\theta}_{e_{i_j}}$. 
 
 Let $\Upsilon_{\infty}$ be the subgraph of $\Upsilon$ whose vertices are exactly the hyperideal vertices of $Q_n$, and whose edges are exactly the edges of $Q_n$ lying completely outside of $\overline{\h}$. By assumption $\Upsilon_\infty$ is not empty; choose one of its connected components and consider $\gamma$ the boundary of its regular neighborhood in $\R^2$. Since $\Upsilon$ contains no cycles, $\gamma$ must be connected (otherwise the regular neighborhood of $\Upsilon_\infty$ would contain a cycle). Proposition 5 of \cite{bonbao} implies that $\sum_\gamma\tilde{\theta}_e>2\pi$. If $\gamma$ only intersects the old edges $e_1,\dots,e_k$, then we have concluded. Suppose then that instead $\gamma$ intersects $e_i'$. We distinguish two cases, based on whether both endpoints of $e_i'$ are in $\Upsilon_\infty$ or not.
 
 If both endpoints of $e_i'$ are in $\Upsilon_\infty$, still $e_i'$ cannot be contained in $\Upsilon_\infty$ by construction; moreover its endpoints cannot be contained in the same connected component because $\Upsilon$ contains no cycles. One endpoint of $e_i'$ is contained in a connected component $U$ of $\Upsilon_\infty$ whose regular neighborhood is bounded by $\gamma$, the other in a component $U'$ whose regular neighborhood is bounded by $\gamma'$. Let $\gamma''$ be the boundary of the regular neighborhood of $U\cup U'\cup e_i'$. Then
 
 $$\sum_{\gamma''}\tilde{\theta}_e=\sum_{\gamma'}\tilde{\theta}_e+\sum_{\gamma}\tilde{\theta}_e-2\tilde{\theta}_{e_i'}>4\pi-2\tilde{\theta}_{e_i'}>2\pi.$$
 
 If instead one endpoint $v$ of $e_i'$ is not in $\Upsilon_\infty$, denote with $\gamma'$ the boundary of a regular neighborhood $U'$ of $v$ in $\R^2$, and with $\gamma''$ the boundary of a regular neighborhood of $U\cup U'\cup e_i'$. Then
 
 $$\sum_{\gamma''}\tilde{\theta}_e=\sum_{\gamma}\tilde{\theta}_e+\sum_{\gamma'}\tilde{\theta}_e-2\tilde{\theta}_{e_i'}>\sum_{\gamma}\tilde{\theta}_e>2\pi.$$
 
 where the first inequality is because the external angles of hyperbolic polyhedra must satisfy the triangular inequality, thus $\sum_{\gamma'}\tilde{\theta}_e-2\tilde{\theta}_{e_i'}>0$.
 
 We can repeat this process, modifying $\gamma$ while keeping $\sum_\gamma\tilde{\theta}_e>2\pi$, until $\gamma$ intersects only $e_1,\dots,e_k$, obtaining the desidered inequality. Figure \ref{fig:proof} exemplifies this process in a particular case.
 
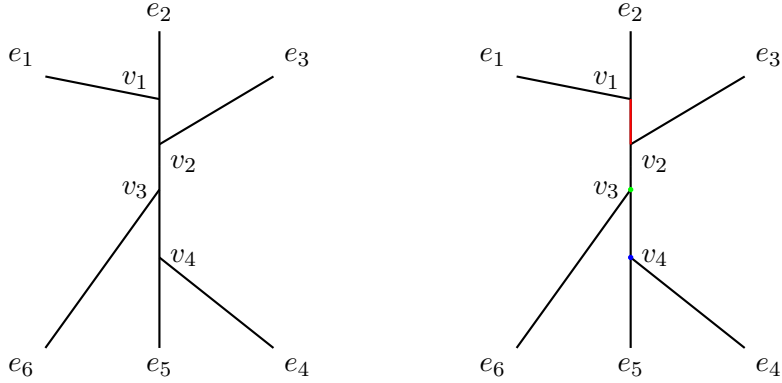
\begin{figure}
 \centering
 \begin{minipage}{.48\textwidth}
  \centering
  \begin{tikzpicture}[scale=0.6]
\centering

\draw[thick] (3,3) node[above]{$e_2$}--(3,0) ;
\draw[thick] (3,0)--(3,-4) node[below]{$e_5$};
\draw[thick] (0.5,2) node[above left]{$e_1$}--(3,1.5) node[above left]{$v_1$};
\draw[thick] (5.5,2) node[above right]{$e_3$}--(3,0.5) node[below right]{$v_2$};
\draw[thick](0.5,-4) node[below left]{$e_6$}--(3,-0.5) node[left]{$v_3$};
\draw[thick] (3,-2)  node[right]{$v_4$}--(5.5,-4) node[below right]{$e_4$};
\end{tikzpicture}
 \end{minipage}
  \begin{minipage}{.48\textwidth}
   \centering   
  \begin{tikzpicture}[scale=0.6]
\centering

\draw[thick] (3,3) node[above]{$e_2$}--(3,0) ;
\draw[thick] (3,0)--(3,-4) node[below]{$e_5$};
\draw[thick] (0.5,2) node[above left]{$e_1$}--(3,1.5) node[above left]{$v_1$};
\draw[thick] (5.5,2) node[above right]{$e_3$}--(3,0.5) node[below right]{$v_2$};
\draw[thick](0.5,-4) node[below left]{$e_6$}--(3,-0.5) node[left]{$v_3$};
\draw[thick] (3,-2)  node[right]{$v_4$}--(5.5,-4) node[below right]{$e_4$};
\draw[thick, red] (3,1.5)--(3,0.5);
\draw[thick, green, fill=green] (3,-0.5) circle[radius=1pt];
\draw[thick, blue, fill=blue] (3,-2) circle[radius=1pt];
\end{tikzpicture}
  \end{minipage}
  \caption*{The $1$-skeleton $\Upsilon$. The subgraph $\Upsilon_\infty$ is comprised of edge $v_1v_2$ and vertex $v_4$. The edge $v_1v_2$ in $Q_n$ will lie outside $\overline{\h}$, the vertex $v_4$ will be hyperidal and $v_3$ will be real.}
  
  \begin{minipage}{.48\textwidth}
  
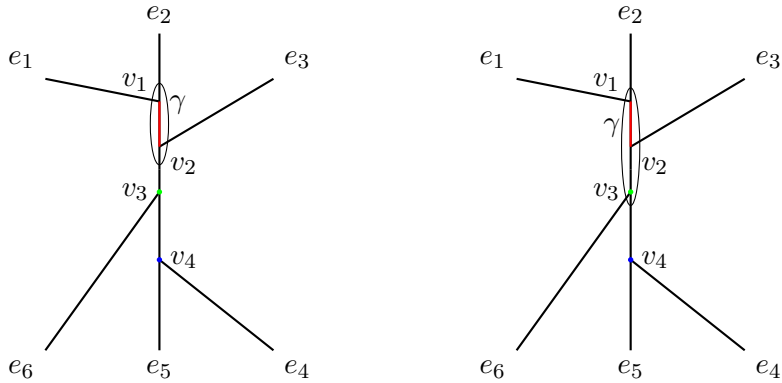
\centering  
  \begin{tikzpicture}[scale=0.6]
\centering

\draw[thick] (3,3) node[above]{$e_2$}--(3,0) ;
\draw[thick] (3,0)--(3,-4) node[below]{$e_5$};
\draw[thick] (0.5,2) node[above left]{$e_1$}--(3,1.5) node[above left]{$v_1$};
\draw[thick] (5.5,2) node[above right]{$e_3$}--(3,0.5) node[below right]{$v_2$};
\draw[thick](0.5,-4) node[below left]{$e_6$}--(3,-0.5) node[left]{$v_3$};
\draw[thick] (3,-2)  node[right]{$v_4$}--(5.5,-4) node[below right]{$e_4$};
\draw[thick, red] (3,1.5)--(3,0.5);
\draw[thick, green, fill=green] (3,-0.5) circle[radius=1pt];
\draw[thick, blue, fill=blue] (3,-2) circle[radius=1pt];
\draw[thin] (3,1) ellipse(.2cm and .9cm) node[above right]{$\gamma$};
\end{tikzpicture}
  \end{minipage}
\begin{minipage}{.48\textwidth}
 \centering
  \begin{tikzpicture}[scale=0.6]
\centering

\draw[thick] (3,3) node[above]{$e_2$}--(3,0) ;
\draw[thick] (3,0)--(3,-4) node[below]{$e_5$};
\draw[thick] (0.5,2) node[above left]{$e_1$}--(3,1.5) node[above left]{$v_1$};
\draw[thick] (5.5,2) node[above right]{$e_3$}--(3,0.5) node[below right]{$v_2$};
\draw[thick](0.5,-4) node[below left]{$e_6$}--(3,-0.5) node[left]{$v_3$};
\draw[thick] (3,-2)  node[right]{$v_4$}--(5.5,-4) node[below right]{$e_4$};
\draw[thick, red] (3,1.5)--(3,0.5);
\draw[thick, green, fill=green] (3,-0.5) circle[radius=1pt];
\draw[thick, blue, fill=blue] (3,-2) circle[radius=1pt];
\draw[thin] (3,0.5) ellipse(.2cm and 1.3cm) node[above left]{$\gamma$};
\end{tikzpicture}
\end{minipage}
\caption*{We start with $\gamma$ encircling $v_1v_2$. We extend $\Gamma$ to encircle $v_3$: doing so increases $\sum_\gamma\tilde{\theta}_e$ because of the triangular inequality.}
\centering
  \begin{tikzpicture}[scale=0.6]
\centering

\draw[thick] (3,3) node[above]{$e_2$}--(3,0) ;
\draw[thick] (3,0)--(3,-4) node[below]{$e_5$};
\draw[thick] (0.5,2) node[above left]{$e_1$}--(3,1.5) node[above left]{$v_1$};
\draw[thick] (5.5,2) node[above right]{$e_3$}--(3,0.5) node[below right]{$v_2$};
\draw[thick](0.5,-4) node[below left]{$e_6$}--(3,-0.5) node[left]{$v_3$};
\draw[thick] (3,-2)  node[right]{$v_4$}--(5.5,-4) node[below right]{$e_4$};
\draw[thick, red] (3,1.5)--(3,0.5);
\draw[thick, green, fill=green] (3,-0.5) circle[radius=1pt];
\draw[thick, blue, fill=blue] (3,-2) circle[radius=1pt];
\draw[thin] (3,-.25) ellipse(.2cm and 2.1cm) node[above left]{$\gamma$};
\end{tikzpicture}
\caption{We finish by extending $\gamma$ to encircle $v_4$ as well. This way $\gamma$ intersects $e_1$ through $e_6$ and $\sum_\gamma \tilde{\theta}_e>2\pi$ as requested.}\label{fig:proof}
\end{figure}

 \textbf{Case 2.} Some edges $e_i^\infty$ intersect $\h$, while at least some other $e_j^\infty$ is tangent to $\partial\h$; however every edge tangent to $\partial\h$ is in the same edge of $P^*$.
 
 By renumbering if necessary suppose that $e_1^n,\dots,e_{k'}^n$ converge to edges intersecting $\h$ and the remainder do not.
 
 In this case we employ the same trick of Lemma \ref{lem:anglesum}: we truncate $Q_n$ and we double along the truncation face to fall back into Case 1. This way we show that $\sum_i^{k'}\tilde{\theta}_i^n>\pi$ which once again contradicts Lemma \ref{lem:anglesum} since the angles are decreasing.
\end{proof}

 \begin{prop:1} 
 Let $P$ be either a proper or almost proper polyhedron with $1$-skeleton $\Gamma$ with no ideal vertices and some real vertices. Then there exists a generalized hyperbolic polyhedron $P^*$ with the following properties:
 \begin{itemize}
  \item $\vol(P^*)>\vol(P)$;
  \item $P^*$ is either proper or almost proper;
  \item $P^*$ has at most the same number of real vertices as $P$;
  \item $P^*$ either has fewer vertices, fewer real vertices or fewer proper vertices than $P$;
  \item if $P^*$ has ideal vertices, then it is proper.
 \end{itemize}
Furthermore, the $1$-skeleton of $P^*$ can be obtained from $\Gamma$ via a finite sequence of the following moves:
   \begin{enumerate}[label=(\roman*)]
    \item an edge of $\Gamma$ collapses to a vertex;
    \item a face of $\Gamma$ collapses to an edge.
   \end{enumerate}
 \end{prop:1}

 \begin{proof}
 The strategy to obtain the polyhedron $P^*$ is to deform $P$ by decreasing all angles (hence, increasing the volume) until it is no longer possible.
 
 If $P$ has only real vertices, then for an appropriate $\lambda>1$ the polyhedron $\Phi_\lambda^v(P)$ (for $v\in P$) is a proper polyhedron with $1$-skeleton $\Gamma$, at least an ideal vertex and larger volume (since it clearly contains $P$), hence it satisfies the conditions in the thesis. 
 
 Suppose now that $P$ has a hyperideal vertex $v$, and let $\vec{\theta}=(\theta_1,\dots,\theta_k)$ be the dihedral angles of all proper edges of $P$ (if $P$ is proper, all its edges are proper). Then Corollary \ref{cor:loccoord} or Corollary \ref{cor:locbad} tell us that for $t$ in a neighborhood of $1$ there is a continuous family of polyhedra $P_t$ (for now, only defined up to isometry) with $P_1=P$, $1$-skeleton $\Gamma$, dihedral angles of proper edges equal to $t\vec{\theta}$ and the same almost proper vertices of $P$. By the Schl\"afli identity, the volume of $P_t$ increases as $t$ decreases. Up to a small perturbation of $\vec{\theta}$ that decreases all angles, we can assume that the path $t\vec{\theta}$ intersects the hyperplanes $\sum_{i\in I}\theta_i=k\pi$ one at a time, where $I$ is any possible subset of $\{1,\dots,k\}$. Let $t_*$ the infimum of all $t$'s such that $P_t$ is defined. 
 
 If $t_*=0$, then for $t$ very close to $0$ the polyhedron $P_t$ is hyperideal; since $P$ has some real vertices then $P_t$ has more hyperideal vertices than $P$ and thus satisfies all the conditions of the thesis. Suppose then $t_*>0$. We prove that in this case $P_t$ has an accumulation point as $t\ra t_*$.
 
 \begin{lem}
  There is a suitable choice of family $P_t$ with $1$-skeleton $\Gamma$ and angles $t\vec{\theta}$ which has an accumulation point $Q$ that is a non-degenerate projective polyhedron.
 \end{lem}
\begin{proof}
We need to prove that, up to a suitable choice of isometry class for $P_t$, there is a converging subsequence, and its limit is a non-degenerate polyhedron contained in an affine chart (i.e. it does not contain any line).
 Choose any subsequence $P_n:=P_{t_n}$ with $t_n\ra t_*$.
 Choose $P_n$ in the isometry class of polyhedra with angles $t_n\vec{\theta}$ in such a way that a certain hyperideal vertex $v$ of $P_n$ is fixed (we assume there is one because of the remark at the beginning of the proof). 
 
 Consider now $A_n:=P_n\cap \Pi_v$: it is a sequence of polygons satisfying the hypotheses of Lemma \ref{lem:conv2} (when viewed as subsets of $\Pi_v\cong \mathbb{H}^2$); then up to isometry of $\mathbb{H}^2$ and subsequence they converge to a non-degenerate polygon $A$. This shows that we can change each $P_n$ by isometry so that $A_n$ converges to $A$. Notice that \emph{a priori} this choice of isometry for each $P_n$ might make it so that they are not all contained in the same affine chart of $\rp$.
 
 By compactness of $(\mathbb{RP}^3)^*$ the polyhedra $P_n$ are going to have a subsequence converging to some convex set $Q\subseteq\mathbb{RP}^3$.
 
 Then $Q$ must contain the pyramid with base $A$ and vertex $v$, hence it is non-degenerate.
  
 We need to show that $Q$ is contained in an affine chart. Suppose by contradiction that $Q$ contains a projective line $l$. If $v\in l$ then there must be $w_n\in P_n$ different from $v$ but converging to $v$, and then the distance between $\Pi_v$ and $\Pi_{w_n}$ must converge to $0$. This implies that the thickness of $P_n$ (i.e. the radius of the largest ball contained in $P_n$) must also converge to $0$; this implies that $\vol(P_n)\ra 0$ by \cite[Proposition 4.2]{miyamoto} which is a contradiction.
 If instead $v\notin l$ then we can find a vertex $w_n$ in $P_n$ arbitrarily far (in the Euclidean sense) from $\overline{\h}$, and then $\overrightarrow{vw_n}$ does not intersect $\overline{\h}$ which is absurd.
\end{proof}

We pass to a further subsequence, if needed, so that all vertices of $P_n$ converge.

We now distinguish $3$ cases.

 \textbf{Case 1.} If $Q$ is a generalized hyperbolic polyhedron without ideal vertices, then we define $P^*:=Q$.
 
 We need to check several properties of $P^*$: we do so in the same order we listed them in the statement. 
 
 \begin{itemize}
  \item  Since $P^*$ has compact truncation, Lemma \ref{lem:continuity} implies that $\vol(P_t)\ra \vol(P^*)$ increasingly, hence $\vol(P^*)>\vol(P)$. 
  \item Since $P^*$ is a limit of proper or almost proper polyhedra, it must be proper or almost proper itself, as being proper or almost proper is a closed condition.
  \item By Lemma \ref{lem:hypvert} an ideal or hyperideal vertex of $P_n$ cannot become ideal or real in $P^*$, hence $P^*$ has at most the same number of real vertices as $P$.
  \item If $P^*$ had the same number of vertices (hence the same $1$-skeleton), no ideal vertices and the same almost proper vertices, then the dihedral angles of proper edges would be local coordinates around $P^*$; this would imply that actually $P_t\ra P^*$ (since all limit points of $P_t$ must have the same angles, hence be locally the same) and we could extend the path $P_t$. This would contradict the fact that $t_*$ is minimal.
 \end{itemize}

 We need to show then that the $1$-skeleton of $P^*$ can be obtained by a sequence of edge or face collapses. This is done by applying Lemma \ref{lem:coll2}. 
 
 This conclude the proof in the case where $Q$ is a generalized hyperbolic polyhedron without ideal vertices.
 
 \textbf{Case 2.} If $Q$ is a projective polyhedron without ideal vertices (i.e. all its vertices are not on the sphere at infinity $\partial\h$), then it is a generalized hyperbolic polyhedron. 
 
 To see this, remember that we only need to show that every edge of $Q$ intersects $\h$. Suppose that an edge $e$ of $Q$ is instead tangent to $\partial \h$. By assumption neither of the endpoints of $e$ can lie on $\partial\h$, hence they must lie outside of $\overline{\h}$. First we prove that there is only one sequence of edges $e_n$ converging to $e$ (or even a subset of $e$). If there was another sequence of edges $e'_n$ of $P_n$ converging to a subset of $e$, then by the discussion following Lemma \ref{lem:anglesum} both of its endpoint would have to converge to points outside $\overline{\h}$. Then the lines connecting these endpoints to the endpoints of $e_n$ must, for $n$ big enough, lie outside $\h$ which is a contradiction. Therefore there is only one edge $e_n$ of $P_n$ converging to $e$, and let $F_n,G_n$ be the two faces containing $e_n$ and converging to $F,G$ faces of $Q$ containing $e$. If one of $F$ or $G$ was tangent to $\partial\h$ then we would have another contradiction as some other edge would lie outside $\overline{\h}$. Therefore $F,G$ must be contained in two hyperbolic planes intersecting with dihedral angle $0$, and the angle between $F_n$ and $G_n$ would converge to $0$ which contradicts the way we chose the angles of $P_n$.
 
 Therefore, we can once again define $P^*:=Q$. The fact that $P^*$ satisfies the thesis is exactly the same as before.
 
 \textbf{Case 3:} $Q$ is a projective polyhedron with some ideal vertices.
 
 The problem in this case is that it could happen that $\vol(P_n)$ does not converge to $\vol(Q)$. In this case we need to modify the sequence $P_n$ to arrive at some other polyhedron $Q'$.
 
 For any $v$ ideal vertex of $Q$ let (as in the proof of Lemma \ref{lem:anglesum}) $K_v\subseteq\Gamma$ be the union of all edges and vertices collapsing to $v$. Let $e_1,\dots,e_k$ be the edges of $\Gamma$ with exactly one endpoint in $K_v$ (notice that since $Q$ is non-degenerate there are at least $3$ such edges), let $e_1^n,\dots,e_k^n$ be the corresponding edges of $P_n$, and $e_1^\infty,\dots,e_n^\infty$ their limit in $Q$.
 
 Let $\overline{n}$ be big enough that all the vertices of $P_{\overline{n}}$ converging to $v$ are very close to $\partial \h$ (in the Euclidean distance), while every other vertex is farther.
 
 We further distinguish two cases.
 
 \textbf{Case 3a.} The ideal vertex of $Q$ is proper (i.e. it is not contained in the dual plane of any hyperideal vertex of $Q$).
 
 Notice that in this case we can apply the same reasoning of Case 2 to get that $Q$ is a generalized hyperbolic polyhedron.
  
 We have shown Lemma \ref{lem:anglesum} that in this case $\lim_{n\ra\infty}\sum_i\theta_{e_i^n}$ is a multiple of $\pi$, therefore $Q$ has exactly one ideal vertex because of the way we perturbed $\vec{\theta}$.
  
 Then there is a hyperbolic plane $\Pi$ delimiting the half-spaces $H_1$ and $H_2$ such that $H_1$ contains, for every $n\geq \overline{n}$, exactly the vertices of $P_n$ converging to $v$, while $H_2$ contains every other vertex of $P_n$ and every truncation plane (notice: not just truncation faces). For simplicity we can assume that $\Pi$ is the dual plane to some hyperideal point close to $v$ (so that $\Pi$ is almost orthogonal to $e_1^n,\dots,e_k^n$). Up to an isometry we can make it so that $\Pi$ is an equatorial plane (i.e. one containing $0\in\h\subseteq \R^3$). Let $\overrightarrow{a}$ be the unit normal vector to $\Pi$ pointing toward $H_1$. Recall that $\Psi_{\overrightarrow{b}}$ is the translation of $\R^3$ in the direction $\overrightarrow{b}$. 
 
 For $n$ big enough, $Q\cap H_2$ is compact and close to $P_n\cap H_2$, hence their volumes are also close. Then for any $\delta>0$ there is a $\lambda$ small enough that $\vol\left(\Psi_{\lambda\overrightarrow{a}}(Q)\cap H_2\right)>\vol\left(Q\cap H_2\right)-\frac{1}{2}\delta$, which implies that $$\vol\left(\Psi_{\lambda\overrightarrow{a}}(P_n)\cap H_2\right)>\vol\left(P_n\cap H_2\right)-\delta$$ for every $n$ big enough. It is important to notice that $\lambda$ does not depend on $n$, only on $\delta$.
 
 Then we define $P^*$ to be $\Psi_{\lambda'\overrightarrow{a}}(P_n)$ for some $n$ and $\lambda'<\lambda$. For $n$ sufficiently large and an appropriate $\lambda'$, there is some vertex of $P_n$ that becomes ideal. Moreover since every vertex that is close to $v$ is real by Lemma \ref{lem:hypvert}, this must happen before any edge of $P_n$ gets pushed out of $\overline{\h}$. This implies that $P^*$ is a generalized hyperbolic polyhedron.
 
 We prove that $P^*$ satisfies all the conditions of the thesis, in the same order.
 
 \begin{itemize}
  \item Clearly $\vol(P^*\cap H_1)>\vol(P_n\cap H_1)$, since $P_n\cap H_1\subseteq P^*\cap H_1$. Furthermore we chose $\lambda'<\lambda$ such that $\vol(P^*\cap H_2)>\vol(P_n\cap H_2)-\delta$. Therefore $\vol(P^*)>\vol(P_n)-\delta$, which implies that $\vol(P^*)>\vol(P)$ for $\delta$ small enough.
  \item The translation vector $\overrightarrow{a}$ is contained in the tangent cones of all hyperideal vertices of $P_n$ (see Figure \ref{fig:unpropertot}), therefore $P^*$ is proper by Lemma \ref{lem:unproper}.
  \item $P_n$ has at most the same number of real vertices of $P$ by Lemma \ref{lem:hypvert}; then $P^*$ has some more ideal vertices, hence fewer real vertices.
  \item $P^*$ has exactly the same number of vertices of $P$ but at least one fewer real vertices, as we noted in the preceding point.
  \item we have noted that $P^*$ is proper.
 \end{itemize}
Furthermore $P^*$ has the same $1$-skeleton as $P_n$ hence the same $1$-skeleton as $P$.
 
 \begin{figure}
  \centering
 \begin{tikzpicture}
  \draw (1,1) circle[radius=3cm];
  \draw[fill=black] (5,5) node[above]{$v$} circle[radius=1pt];
  \draw (5,5)--(-1,3.62);
  \draw (5,5)--(3.64,-1);
  \draw (0.33,3.93)node[above]{$\Pi_v$}--(3.93,0.25);
  \draw[red] (-1.8,2.1) node[above left,black]{$\Pi$} --(3.8,-.1);
  \draw[fill=red, red] (4.8,4.6) node[right,black]{$\Psi_{\lambda\overrightarrow{a}}(v)$} circle[radius=1pt];
  \draw[red,->] (1,1)--(.8,.6) node[black, below]{$\lambda\vec{a}$};
 \end{tikzpicture}
 \caption{Because $\Pi_v$ is completely contained in $H_2$, the translation by $\lambda\vec{a}$ sends $v$ into its tangent cone.}\label{fig:unpropertot}
 \end{figure}
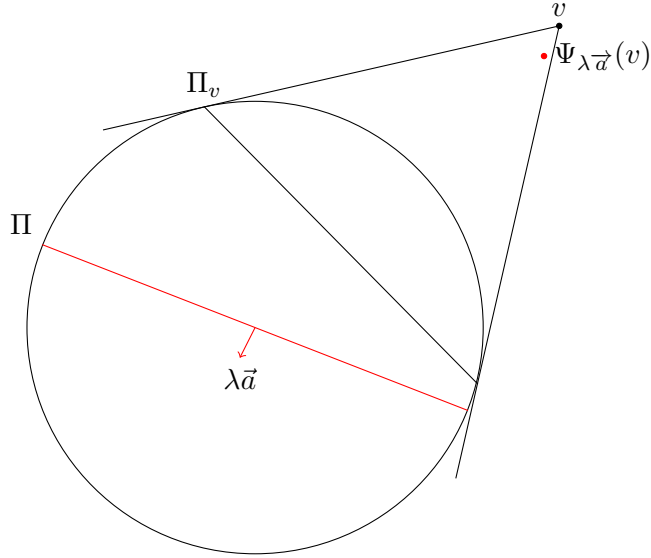

 \textbf{Case 3b.} The ideal vertex of $Q$ is almost proper.
 
 This case is almost the same as case 3a; we just need to be careful about the truncating plane containing the almost proper vertex.

 Once again by Lemma \ref{lem:anglesum} we have that $\lim_{n\ra\infty}\sum_{i=1}^{k'}\theta_{e_i^n}$ is a multiple of $\pi$, therefore $Q$ has exactly one ideal vertex.

 Let $v$ be the almost proper vertex of $Q$ and let $w$ be the vertex of $Q$ such that $v\in\Pi_w$. Since $w$ must be hyperideal, there is a unique vertex $w_n\in P_n$ converging to it. Now fix $n$ big enough. As before we can find a plane $\Pi$ that divides $\h$ in $H_1,H_2$ with $H_1$ containing every vertex that converges to $v$ and $H_2$ containing every other vertex and the dual plane to every hyperideal vertex of $P_n$ other than $w_n$. Furthermore we can choose $\Pi$ passing through $w_n$, so that $\Pi$ and $\Pi_{w_n}$ are orthogonal.
  
 With an isometry we fix $\Pi_{w_n}$ to be an equatorial plane and $\Pi_v$ is an equatorial plane orthogonal to it (notice that in this particular case of equatorial planes, being orthogonal in $\h$ is the same as being orthogonal in $\R^3$). Having $\Pi_{w_n}$ being equatorial means that $w_n$ is a point at infinity in $\rp$.
 
 Let $\overrightarrow{a}$ be the unit normal vector to $\Pi$ pointing towards $H_1$ and $\overrightarrow{b_n}$ be the unit normal vector to $\Pi_{w_n}$ pointing towards the truncation of $P_n$.
 
 Then the polyhedron $\Psi_{\lambda(\overrightarrow{a}+\epsilon\overrightarrow{b_n})}(P_n)$, for $n$ big enough and $\lambda$ and $\epsilon$ small enough, satisfies all the conditions required for $P^*$. The proof is almost the same as in case $3a$: the only additional detail to check is that any almost proper vertex lying on $\Pi_{w_n}$ becomes proper. To see this, notice that $\Psi_{\lambda(\overrightarrow{a}+\epsilon\overrightarrow{b})}$ does not move $w_n$ since it is a point at infinity (hence, leaves $\Pi_{w_n}$ fixed) and pushes every vertex away from $w_n$.

 \end{proof}

\bibliographystyle{plain}
\bibliography{Bibliography}

\def\cprime{$'$}
\begin{thebibliography}{10}

\bibitem{andreev}
E.~M. Andreev.
\newblock {On convex polyhedra in Lobachevskii spaces}.
\newblock {\em Matematicheskii Sbornik}, 123(3):445--478, 1970.

\bibitem{bonbao}
X.~Bao and F.~Bonahon.
\newblock Hyperideal polyhedra in hyperbolic 3-space.
\newblock {\em Bulletin de la Soci{\'e}t{\'e} math{\'e}matique de France},
  130(3):457--491, 2002.

\bibitem{maxvolconj}
G.~Belletti.
\newblock A maximum volume conjecture for hyperbolic polyhedra.
\newblock {\em arXiv:2002.01904}, 2020.

\bibitem{diaz}
R.~Díaz.
\newblock Non-convexity of the space of dihedral angles of hyperbolic
  polyhedra.
\newblock {\em Comptes Rendus de l'Académie des Sciences - Series I -
  Mathematics}, 325(9):993 -- 998, 1997.

\bibitem{fle}
H.~Fleischner.
\newblock The uniquely embeddable planar graphs.
\newblock {\em Discrete Mathematics}, 4(4):347--358, 1973.

\bibitem{miln}
J.W. Milnor.
\newblock {\em {Collected papers. 1. Geometry}}.
\newblock Publish or Perish, 1994.

\bibitem{miyamoto}
Y.~Miyamoto.
\newblock On the volume and surface area of hyperbolic polyhedra.
\newblock {\em Geometriae Dedicata}, 40(2):223--236, 1991.

\bibitem{mont}
G.~Montcouquiol.
\newblock Deformations of hyperbolic convex polyhedra and cone-3-manifolds.
\newblock {\em Geometriae Dedicata}, 166(1):163--183, 2013.

\bibitem{rivinvol}
I.~Rivin.
\newblock Euclidean structures on simplicial surfaces and hyperbolic volume.
\newblock {\em Annals of mathematics}, 139(3):553--580, 1994.

\bibitem{rivin}
I.~Rivin.
\newblock A characterization of ideal polyhedra in hyperbolic 3-space.
\newblock {\em Annals of mathematics}, pages 51--70, 1996.

\bibitem{rivhodg}
I.~Rivin and C.~D. Hodgson.
\newblock A characterization of compact convex polyhedra in hyperbolic 3-space.
\newblock {\em Inventiones mathematicae}, 111(1):77--111, 1993.

\bibitem{steinitz}
E.~Steinitz.
\newblock Polyeder und raumeinteilungen.
\newblock {\em Encyk der Math Wiss}, 12:38--43, 1922.

\bibitem{thurston}
W.~Thurston.
\newblock {\em The geometry and topology of three-manifolds}.
\newblock Princeton University Princeton, NJ, 1979.

\bibitem{ushi}
A.~Ushijima.
\newblock A volume formula for generalised hyperbolic tetrahedra.
\newblock In {\em Non-Euclidean geometries}, pages 249--265. Springer, 2006.

\bibitem{vesego}
A.~Vesnin and A.~Egorov.
\newblock {Ideal right-angled polyhedra in Lobachevsky space}.
\newblock {\em preprint arXiv:1909.11523}.

\bibitem{weiss}
H.~Weiss.
\newblock The deformation theory of hyperbolic cone--3--manifolds with
  cone-angles less than 2$\pi$.
\newblock {\em Geometry \& Topology}, 17(1):329--367, 2013.

\end{thebibliography}

\address
\end{document}